\pdfoutput=1
\documentclass[11pt]{article}

\usepackage[utf8]{inputenc}
\usepackage[pdftex]{graphicx,xcolor}
\usepackage{array,bookmark,fullpage}
\usepackage{microtype}
\usepackage{amsmath,amssymb,amsthm,mathtools,bm,tikz-cd,stmaryrd,accents,enumitem}
\usepackage[nottoc]{tocbibind}
\usepackage{hyperref}
\hypersetup{pdftex,colorlinks=true,allcolors=blue,unicode,psdextra}
\usepackage[capitalize]{cleveref}

\usepackage{mymacros}

\mathlig{//}{{\sslash}}

\newcommand*\uA{\undertilde{A}}
\newcommand*\MALG{\mathrm{MALG}}
\newcommand*\INV{\mathrm{INV}}
\newcommand*\EINV{\mathrm{EINV}}
\newcommand*\fin{\mathrm{fin}}
\newcommand*\ap{\mathrm{ap}}
\DeclareMathOperator\dom{dom}
\DeclareMathOperator\rng{rng}

\DeclarePairedDelimiter\den{\llbracket}{\rrbracket}
\newcommand{\relmiddle}[1]{\nonscript\;\middle#1\nonscript\;}

\let\defn\textbf

\begin{document}

\title{Decompositions and measures on countable Borel equivalence relations}
\author{Ruiyuan Chen\thanks{Research partially supported by NSERC PGS D.}}
\date{}
\maketitle

\begin{abstract}
We show that the uniform measure-theoretic ergodic decomposition of a countable Borel equivalence relation $(X, E)$ may be realized as the topological ergodic decomposition of a continuous action of a countable group $\Gamma \curvearrowright X$ generating $E$.  We then apply this to the study of the cardinal algebra $\mathcal K(E)$ of equidecomposition types of Borel sets with respect to a compressible countable Borel equivalence relation $(X, E)$.  We also make some general observations regarding quotient topologies on topological ergodic decompositions, with an application to weak equivalence of measure-preserving actions.
\renewcommand\thefootnote{}
\end{abstract}

\section{Introduction}

In this paper, we study several related constructions on a countable Borel equivalence relation.

In \cref{sec:topmeasdecomp}, we study the relation between two different notions of ergodic decomposition.  An action of a group $G$ via homeomorphisms on a Polish space $X$ is \defn{minimal} if $X \ne \emptyset$ and each orbit is dense; the \defn{topological ergodic decomposition} of an arbitrary action $G \curvearrowright X$ is the standard Borel decomposition of $X$ into minimal invariant $G_\delta$ subsets.  A countable Borel equivalence relation $E$ on a standard Borel space $X$ is \defn{uniquely ergodic} if it admits a unique ergodic invariant probability measure; the \defn{measure-theoretic ergodic decomposition} of an arbitrary $(X, E)$ is the standard Borel decomposition of $X$ into $E$-invariant, uniquely ergodic pieces.

We show that for a countable Borel equivalence relation $(X, E)$, generated by a countable Borel group action $\Gamma \curvearrowright X$, the measure-theoretic ergodic decomposition of $E$ may be realized \emph{via} the topological ergodic decomposition, with respect to some Polish topology on $X$ making the $\Gamma$-action continuous.  Moreover, we may pick the topology so as to include in the decomposition not only the invariant ergodic probability measures, but also all invariant ergodic $\sigma$-finite measures which are regular with respect to the topology, where ``regular'' means in the weak sense that there is some open set with finite positive measure.  Here is a rough statement of the result; see \cref{thm:topmeasdecomp}.

\begin{theorem}
\label{thm:intro-topmeasdecomp}
Let $(X, E)$ be a countable Borel equivalence relation, induced by a countable Borel group action $\Gamma \curvearrowright X$.  For cofinally many Polish topologies on $X$ inducing the Borel structure and making the $\Gamma$-action continuous, we have the following:
\begin{itemize}
\item  Each component of the topological ergodic decomposition of the action $\Gamma \curvearrowright X$ admits, up to scaling, at most one $E$-invariant $\sigma$-finite measure which is regular with respect to the topology; and such a measure is $E$-ergodic (if it exists).
\item  The set $R$ of components admitting such a measure, as well as the set $P \subseteq R$ of components admitting an $E$-invariant probability measure, are Borel.
\item  We have a Borel bijection between $R$ and the space of $E$-invariant ergodic regular $\sigma$-finite measures modulo scaling, taking a component in $R$ to the unique such measure on that component.  This restricts to a Borel bijection between $P$ and the space $\EINV_E$ of $E$-invariant ergodic probability measures, yielding the usual measure-theoretic ergodic decomposition.
\end{itemize}
\end{theorem}

In \cref{sec:kl}, we study the following canonical algebraic structure associated to a \emph{compressible} countable Borel equivalence relation $(X, E)$.  A \defn{cardinal algebra} \cite{Tar} is a set equipped with a countable addition operation, satisfying certain axioms motivated by cardinal arithmetic.  Cardinal algebras appear naturally in the study of group actions and paradoxical decompositions (see e.g., \cite{Ch1}), as well as the classification of Borel equivalence relations \cite{KMd}.  An example belonging to both of these contexts is the \defn{algebra $\@K(E)$ of equidecomposition types} of Borel sets $A \subseteq X$ with respect to a compressible equivalence relation $E$ on $X$.  We show that several well-known results about countable Borel equivalence relations translate to nice algebraic properties of $\@K(E)$ and a related ``completion'' algebra $\@L(E)$; see \cref{thm:k,thm:l,thm:l-duality}.

\begin{theorem}
\label{thm:intro-kl}
Let $(X, E)$ be a compressible countable Borel equivalence relation.
\begin{itemize}
\item  $\@K(E)$ is a cardinal algebra with finite meets, countable joins, and real multiples of elements represented by $E$-aperiodic Borel sets $A \subseteq X$, and obeys all Horn axioms involving these operations which hold in the algebra $[0, \infty]$.
\item  The completion $\@L(E) \supseteq \@K(E)$ by adjoining real multiples for \emph{all} elements of $\@K(E)$ can be naturally viewed as a cardinal algebra of $E$-equidecomposition types of Borel real-valued functions on $X$.  $\@L(E)$ has finite meets, countable joins, and real multiples of all elements, and obeys all Horn axioms involving these operations which hold in the algebra $[0, \infty]$.
\item  Homomorphisms $\@K(E) -> [0, \infty]$ preserving the above operations are in canonical bijection with $E$-invariant $E$-ergodic measures; the same holds for $\@L(E)$ in place of $\@K(E)$.
\item  The space $\EINV^\sigma_E$ of all $\sigma$-finite such measures forms a ``dual'' of $\@L(E)$, from which $\@L(E)$ may be recovered as the ``double dual''.
\end{itemize}
\end{theorem}

We begin in \cref{sec:topdecomp} with some general observations regarding topological ergodic decompositions.
Let $G$ be a group acting via homeomorphisms on a Polish space $X$.
By passing to the realm of \defn{quasi-Polish spaces} \cite{deB}, a possibly non-Hausdorff generalization of Polish spaces, we may realize the topological ergodic decomposition of $G \curvearrowright X$ in a canonical way: as the quasi-Polish $T_0$-quotient of the quotient space $X/G$.
We give a simple application of this fact: in the case of the space $\uA(\Gamma, X, \mu)$ of weak equivalence classes of measure-preserving actions of a countable group $\Gamma$ (recently studied by several authors; see \cite{BuK} for a survey), the quasi-Polish topology encodes both the usual compact Hausdorff topology and the weak containment partial ordering.

The appendix contains some technical facts about quasi-Polish spaces which are needed in the rest of the paper.

\paragraph*{Acknowledgments}

I would like to thank Alexander Kechris, Anush Tserunyan, and Matthew de~Brecht for several helpful discussions and comments.
I would also like to thank the anonymous referee for many detailed comments and suggestions, especially for pointing out a flaw in an alternative proof of \cref{thm:k-meetjoin} that I previously gave.

\section{Topological ergodic decompositions}
\label{sec:topdecomp}

Let $X$ be a Polish space and $G$ be a group acting via homeomorphisms on $X$.  Define the preordering $\preccurlyeq_G$ on $X$ by
\begin{align*}
x \preccurlyeq_G y &\iff x \in \-{G \cdot y} \iff \-{G \cdot x} \subseteq \-{G \cdot y}.
\end{align*}
The symmetric part $\approx_G$ of $\preccurlyeq_G$, given by
\begin{align*}
x \approx_G y &\iff x \preccurlyeq_G y \AND y \preccurlyeq_G x \iff \-{G \cdot x} = \-{G \cdot y},
\end{align*}
is a $G_\delta$ equivalence relation, hence smooth (see e.g., \cite{Mil}).  The quotient space $X/{\approx_G}$ is called the \defn{topological ergodic decomposition} of the action $G \curvearrowright X$, and is a standard Borel space, partially ordered by (the quotient of) $\preccurlyeq_G$, and equipped with the projection map $X ->> X/{\approx_G}$ which is invariant Borel and whose fibers are the minimal $G$-invariant $G_\delta$ subsets of $X$.  See e.g., \cite[10.3]{Kgaega}.

Recall that on an arbitrary topological space $X$, the \defn{specialization preordering} is defined by
\begin{align*}
x \lesim y &\iff x \in \-{\{y\}} \iff \forall \text{open } U \subseteq X\, (x \in U \implies y \in U).
\end{align*}
The specialization preordering is a partial order iff $X$ is $T_0$, and is discrete iff $X$ is $T_1$.  Note that the specialization preordering on the quotient space $X/G$ (with the quotient topology) is given by
\begin{align*}
[x]_G \lesim [y]_G &\iff x \in \-{G \cdot y} \iff x \preccurlyeq_G y.
\end{align*}
Hence, the $T_0$-quotient of $X/G$, which we denote by
\begin{align*}
X//G,
\end{align*}
is in canonical bijection with the topological ergodic decomposition $X/{\approx_G}$; and its specialization ordering agrees with $\preccurlyeq_G$.  We henceforth identify $X//G$ with $X/{\approx_G}$ (i.e., we regard the elements of $X//G$ as equivalence classes of elements of $X$, not of $X/G$).

Note that open sets in $X//G$ lift to $G$-invariant open sets $U \subseteq X$; for such $U$, we write the corresponding open set in $X//G$ as
\begin{align*}
U//G \subseteq X//G,
\end{align*}
and similarly for closed sets.

Next, we observe that the quotient topology on $X//G$, though not necessarily Hausdorff, is nonetheless well-behaved.  A \defn{quasi-Polish space} \cite{deB} is a $\*\Pi^0_2$ subset of $\#S^\#N$, where $\#S = \{0 < 1\}$ with the Sierpiński topology ($\{0\}$ closed but not open), and where $\*\Pi^0_2$ means a countable intersection of sets of the form $U \cup F$ with $U$ open and $F$ closed.  Quasi-Polish spaces are closed under countable products, countable disjoint unions, $\*\Pi^0_2$-subsets, and continuous open $T_0$ images; are Polish iff they are regular; can be made Polish by adjoining countably many closed sets to the topology; and induce a standard Borel structure (see \cite{deB} or \cite{Ch} for proofs of these basic facts).  Since the projection $X ->> X//G$ is clearly open, $X//G$ is quasi-Polish, hence standard Borel.  It is easily seen that the Borel structure agrees with the quotient Borel structure induced from $X$, i.e., the usual Borel structure on $X/{\approx_G}$.

Finally, we note that we may consider the following slightly more general context.  Let $X$ be a quasi-Polish space and $E$ be an equivalence relation on $X$ such that the $E$-saturation of every open set $U \subseteq X$ is open.  Then $X/E$ has specialization preorder
\begin{align*}
[x]_E \lesim [y]_E \iff x \in \-{[y]_E};
\end{align*}
we denote this by $x \preccurlyeq_E y$ and its symmetric part by $x \approx_E y$.  So the $T_0$-quotient, denoted
\begin{align*}
X//E \cong X/{\approx_E},
\end{align*}
is the topological ergodic decomposition of $X$ into \defn{$E$-minimal} (meaning each $E$-class is dense) components.  The condition on saturations of open sets ensures that the projection $X ->> X//E$ is open, whence $X//E$ is quasi-Polish (in particular standard Borel).  As before, we identify $X//E$ with $X/{\approx_E}$, and we write $U//E \subseteq X//E$ for the open set corresponding to $E$-invariant open $U \subseteq X$.  For $x \in X$, we put
\begin{align*}
\den{x}_E := [x]_{\approx_E}.
\end{align*}
We recover the earlier case of a $G$-action by taking $E$ to be an orbit equivalence relation $E_G$.

We summarize these observations as follows:

\begin{proposition}
\label{thm:topdecomp}
Let $X$ be a (quasi-)Polish space and $E$ be an equivalence relation on $X$ such that the $E$-saturation of every open $U \subseteq X$ is open.  Then the $T_0$-quotient $X//E$ of the quotient $X/E$ is a quasi-Polish space, and the projection $p : X ->> X//E$ is open with kernel
\begin{align*}
x \approx_E y \iff \-{[x]_E} = \-{[y]_E},
\end{align*}
hence $X//E$ is the topological ergodic decomposition of $(X, E)$.  Moreover, the specialization order on $X//E$, given by
\begin{align*}
\den{x}_E \le \den{y}_E \iff x \preccurlyeq_E y \iff x \in \-{[y]_E}
\end{align*}
(where $\den{x}_E := [x]_{\approx_E}$), is the canonical partial order on the topological ergodic decomposition.  Finally, $\approx_E$ has a Borel selector, i.e., the projection $p$ has a Borel section $s : X//E `-> X$.
\end{proposition}
\begin{proof}
The last statement about the Borel section is a general fact about continuous open maps between quasi-Polish spaces; see e.g., \cite[7.9]{Ch}.  Everything else follows from the above discussion.
\end{proof}

\begin{remark}
De~Brecht has pointed out that conversely, every quasi-Polish space can be expressed as the topological ergodic decomposition (indeed, the quotient) of some Polish space by a Polish group action.  This may be seen as follows: we have a continuous open surjection $q : [0, \infty) ->> \#S$ sending $0$ to $0$ and $(0, \infty) \subseteq [0, \infty)$ to $1$, which is the quotient of $[0, \infty)$ by the multiplicative action of $(0, \infty)$; then $q^\#N : [0, \infty)^\#N ->> \#S^\#N$ is the quotient of the product action of $(0, \infty)^\#N$, and so a $\*\Pi^0_2$ subset $X \subseteq \#S^\#N$ is the quotient $(q^\#N)^{-1}(X)/(0, \infty)^\#N$.
\end{remark}

\subsection{Change of topology}
\label{sec:topdecomp-topchng}

We now record some technical facts, needed in \cref{sec:topmeasdecomp}, concerning the behavior of $X//E$ upon changing the topology of $X$.

Let $X$ be a quasi-Polish space and $E$ be an equivalence relation on $X$ such that the $E$-saturation of every open set is open, as in \cref{thm:topdecomp}.  Let $\tau$ be the topology of $X$.  Since we will be considering other topologies, we write $\den{x}_E^\tau$ for $\den{x}_E \in (X, \tau)//E$ when necessary to avoid confusion; similarly, we write $\approx_E^\tau$ for $\approx_E$, etc.

\begin{lemma}
\label{thm:topdecomp-closed}
Let $F_0, F_1, \dotsc \subseteq X$ be countably many $E$-invariant $\tau$-closed sets, and let $\tau' \supseteq \tau$ be the finer topology obtained by adjoining the $F_i$ to $\tau$.  Then the $E$-saturation of every $\tau'$-open set is $\tau'$-open (as in \cref{thm:topdecomp}), and $(X, \tau')//E$ is $(X, \tau)//E$ with the closed sets $F_i//E$ adjoined to its topology.  (In particular, $(X, \tau)//E = (X, \tau')//E$ as sets, i.e., the topological ergodic decompositions with respect to $\tau, \tau'$ have the same components.)
\end{lemma}
\begin{proof}
A basic $\tau'$-open set is of the form $V = U \cap F_{i_0} \cap \dotsb \cap F_{i_{n-1}}$ for a $\tau$-open set $U$; since the $F_i$ are $E$-invariant, the saturation of $V$ is $[U]_E \cap F_{i_0} \cap \dotsb \cap F_{i_{n-1}}$ which is $\tau'$-open.  This shows that $(X, \tau'), E$ also obey the hypotheses of \cref{thm:topdecomp}, as well as that every $E$-invariant $\tau'$-open set belongs to the topology generated by the $E$-invariant $\tau$-open sets along with the $F_i$; the latter easily implies that $(X, \tau)//E$ and $(X, \tau')//E$ are related in the claimed manner.
\end{proof}

\begin{corollary}
\label{thm:topdecomp-polish}
Under the hypotheses of \cref{thm:topdecomp}, we may adjoin countably many $E$-invariant closed sets to the topology of $X$, such that $X//E$ retains the same elements but becomes Polish.
\end{corollary}
\begin{proof}
Find countably many closed sets $F_i//E \subseteq X//E$, the quotients of $E$-invariant closed sets $F_i \subseteq X$, such that adjoining the $F_i//E$ to the topology of $X//E$ makes $X//E$ Polish (e.g., by embedding $X//E$ as a $\*\Pi^0_2$ subspace of $\#S^\#N$); then adjoin the $F_i$ to the topology of $X$.
\end{proof}

If $\tau \subseteq \tau'$ are two topologies on $X$, both satisfying the hypotheses of \cref{thm:topdecomp}, then the two topological ergodic decompositions are related by a quotient map
\begin{align*}
(X, \tau')//E &->> (X, \tau)//E \\
\den{x}_E^{\tau'} &|-> \den{x}_E^\tau.
\end{align*}
Suppose now that we have a sequence of quasi-Polish topologies $\tau_0 \subseteq \tau_1 \subseteq \dotsb$ on $X$, each with the property that the $E$-saturation of an open set is open.  We then have an inverse sequence
\begin{align*}
\dotsb ->> (X, \tau_2)//E ->> (X, \tau_1)//E ->> (X, \tau_0)//E,
\end{align*}
of which we may take the \defn{inverse limit}
\begin{align*}
\projlim_i (X, \tau_i)//E
&= \{(\den{x_i}_E^{\tau_i})_i \in \prod_i (X, \tau_i)//E \mid \forall i\, (\den{x_{i+1}}_E^{\tau_i} = \den{x_i}_E^{\tau_i})\} \\
&= \{(\den{x_i}_E^{\tau_i})_i \in \prod_i (X, \tau_i)//E \mid \forall i\, (x_{i+1} \approx_E^{\tau_i} x_i)\},
\end{align*}
equipped with the subspace topology (which is quasi-Polish, since equality is $\*\Pi^0_2$).
The union of the $\tau_i$ generates a quasi-Polish topology $\tau$ \cite[Lemma~72]{deB} on $X$, the \defn{join} of the $\tau_i$.  We have the quotient maps $(X, \tau)//E ->> (X, \tau_i)//E$ for each $i$; these induce a comparison map
\begin{align*}
h : (X, \tau)//E &--> \projlim_i (X, \tau_i)//E \\
\den{x}_E^\tau &|--> (\den{x}_E^{\tau_i})_i.
\end{align*}

\begin{lemma}
\label{thm:topdecomp-filtcolim}
Under the above hypotheses, $h$ is a homeomorphism.
\end{lemma}
\begin{proof}
First, we check that $h$ is an embedding.  Let $U//E \subseteq (X, \tau)//E$ be an open set.  Since $\tau$ is generated by $\bigcup_i \tau_i$, we have $U = \bigcup_i U_i$, where each $U_i$ is $\tau_i$-open.  Then $[U_i]_E$ is $E$-invariant $\tau_i$-open, hence descends to an open $[U_i]_E//E \subseteq (X, \tau_i)//E$.  Let
\begin{align*}
V_i &:= \{(\den{x_j}^{\tau_j}_E)_j \in \projlim_j (X, \tau_j)//E \mid \den{x_i}^{\tau_i}_E \in [U_i]_E//E\} \\
&= \{(\den{x_j}^{\tau_j}_E)_j \in \projlim_j (X, \tau_j)//E \mid x_i \in [U_i]_E\}
\end{align*}
be the preimage of $[U_i]_E//E$ under the $i$th projection $\projlim_j (X, \tau_j)//E -> (X, \tau_i)//E$.  Since $U$ is $E$-invariant, $U = \bigcup_i U_i = \bigcup_i [U_i]_E$, whence it is easily seen that $U//E = h^{-1}(\bigcup_i V_i)$.  We have shown that every open set in $(X, \tau)//E$ is the $h$-preimage of an open set in $\projlim_i (X, \tau_i)//E$; since the former space is $T_0$, this means that $h$ is an embedding, as desired.

To check that $h$ is surjective (which is not needed in what follows), we use \cref{thm:opensurj-filtcolim}, with $X_i := (X, \tau_i)$ and $Y_i := (X, \tau_i)//E$.  The Beck--Chevalley condition in the hypotheses of that result amounts to the trivial fact that for $\tau_i$-open $U \subseteq X$, its $E$-saturation is the same whether we regard $U$ as $\tau_i$-open or $\tau_{i+1}$-open.
\end{proof}

\subsection{Weak equivalence of measure-preserving actions}

We give here a simple example of the extra information that may be contained in the quasi-Polish topology on $X//E$.

Let $(X, \mu)$ be a nonatomic standard probability space and $\Gamma$ be a countable group.  The set of measure-preserving actions $a : \Gamma \curvearrowright (X, \mu)$, where two actions are identified if they agree modulo $\mu$-null sets, is denoted
\begin{align*}
A(\Gamma, X, \mu).
\end{align*}
For an action $a : \Gamma \curvearrowright (X, \mu)$, we write $\gamma^a \cdot x := a(\gamma, x)$.  There is a canonical Polish topology on $A(\Gamma, X, \mu)$ (see \cite[II~\S10(A)]{Kgaega}), generated by the maps
$a |-> \gamma^a \cdot B$
to the measure algebra $\MALG_\mu$ of $\mu$, for $\gamma \in \Gamma$ and Borel $B \subseteq X$.  The Polish group $\Aut(X, \mu)$ of measure-preserving automorphisms of $(X, \mu)$ acts continuously on $A(\Gamma, X, \mu)$ via conjugation.  The resulting topological ergodic decomposition
\begin{align*}
\uA(\Gamma, X, \mu) := A(\Gamma, X, \mu)//{\Aut(X, \mu)}
\end{align*}
is the \defn{space of weak equivalence classes} of measure-preserving actions $\Gamma \curvearrowright (X, \mu)$; and the associated preordering $\preccurlyeq$ and equivalence relation $\approx$ on $A(\Gamma, X, \mu)$ are called \defn{weak containment} and \defn{weak equivalence}, respectively.  See \cite[II~\S10(C)]{Kgaega} or \cite[\S2.1]{BuK}.

There is a natural compact Polish topology on $\uA(\Gamma, X, \mu)$, due to Abért--Elek \cite{AE}; various equivalent descriptions of this topology are known (see \cite[\S10.1]{BuK}).  Denote this topology by $\tau$.  The weak containment partial ordering $\preccurlyeq$ is closed as a subset of $\uA(\Gamma, X, \mu)^2$ with the $\tau$-product topology (see \cite[\S10.3]{BuK}).  We also have the quasi-Polish quotient topology on $\uA(\Gamma, X, \mu)$ induced by $A(\Gamma, X, \mu)$; denote this topology by $\sigma$.  (Note that $\sigma$ is not $T_1$, since the specialization order $\preccurlyeq$ is not discrete; see \cite[\S10.3]{BuK}.)

In the theory of topological posets, there is a well-known bijective correspondence between compact Hausdorff spaces equipped with a closed partial order, and the following class of $T_0$-spaces.  A topological space $X$ is \defn{stably compact} if
\begin{itemize}
\item  it is \defn{locally compact}, i.e., every point has a basis of compact neighborhoods;
\item  it is \defn{strongly sober}, i.e., every ultrafilter has a unique greatest limit (in the specialization preorder).
\end{itemize}
(See \cite[VI-6.15]{GHK}, or \cite[VI-6.7]{GHK} for an equivalent definition.)  The \defn{patch topology} on a stably compact space $X$ has basic closed sets consisting of closed sets in $X$ together with compact sets which are upward-closed in the specialization order.  Given an arbitrary topological space $Y$ with a partial order $\le$, the \defn{upper topology} on $Y$ consists of all $\le$-upward-closed open sets.

\begin{theorem}[{\cite[VI-6.18]{GHK}}]
\label{thm:stbcpt}
For any set $X$, there is a bijection
\begin{align*}
\{\text{stably compact topologies $\sigma$ on $X$}\} &\cong \left\{(\tau, \le) \relmiddle| \begin{aligned}
&\tau \text{: compact Hausdorff topology on $X$, } \\
&{\le} \text{: $\tau$-closed partial order on $X$}
\end{aligned}\right\} \\
\sigma &|-> (\text{patch topology},\, \text{specialization order}) \\
\text{upper topology} &\mapsfrom (\tau, \le).
\end{align*}
\end{theorem}

Using this, we have yet another description of the compact Polish topology $\tau$ on $\uA(\Gamma, X, \mu)$:

\begin{proposition}
The quasi-Polish quotient topology $\sigma$ on $\uA(\Gamma, X, \mu)$ induced by $A(\Gamma, X, \mu)$ corresponds, via \cref{thm:stbcpt}, to the compact Polish topology $\tau$ and the weak containment order $\preccurlyeq$.
\end{proposition}
\begin{proof}
It suffices to check that the upper topology of $(\tau, \preccurlyeq)$ is $\sigma$.  By \cite[10.4]{BuK}, every $\sigma$-open set is $\tau$-open, as well as $\preccurlyeq$-upward closed (by definition of the specialization preorder).  For the converse, we use the following description of $\tau$ (see \cite[10.3]{BuK}).  For each $k \in \#N$, finite subset $\Delta \subseteq \Gamma$, and $\vec{r} = (r_{\gamma,i,j})_{\gamma \in \Delta; i,j < k} \in [0, 1]^{\Delta \times k \times k}$, define the upper semicontinuous map
\begin{align*}
f_{\Delta,k,\vec{r}} : A(\Gamma, X, \mu) &--> [0, 1] \\
a &|--> \inf_{\vec{B} \in \MALG_\mu^k} \max_{\gamma \in \Delta; i,j < k} \abs{\mu(\gamma(B_i) \cap B_j) - r_{\gamma,i,j}}.
\end{align*}
Let $\Phi$ denote the set of all such tuples $(\Delta,k,\vec{r})$.  Then $(f_{\Delta,k,\vec{r}})_{\Delta,k,\vec{r}} : A(\Gamma, X, \mu) -> [0, 1]^\Phi$ descends to an order-\emph{reversing} embedding $\uA(\Gamma, X, \mu) -> [0, 1]^\Phi$ onto a closed subposet of $[0, 1]^\Phi$ (see \cite[2.9]{BuK}); and $\tau$ is obtained by pulling back the usual compact Hausdorff topology on $[0, 1]^\Phi$.  It follows that the upper topology of $(\tau, \preccurlyeq)$ is obtained by pulling back the \emph{lower} topology on $[0, 1]^\Phi$ (induced by the open sets $[0, r)$ in $[0, 1]$): indeed, for each closed, $\preccurlyeq$-downward closed $F \subseteq \uA(\Gamma, X, \mu)$, the upward closure of its image in $[0, 1]^\Phi$ is a closed (by compactness), upward-closed set whose pullback to $\uA(\Gamma, X, \mu)$ is $F$ (because $(f_{\Delta,k,\vec{r}})_{\Delta,k,\vec{r}}$ is order-reversing).  Since each $f_{\Delta,k,\vec{r}} : A(\Gamma, X, \mu) -> [0, 1]$ is upper semicontinuous, i.e., continuous with respect to the lower topology on $[0, 1]$, it follows that the upper topology of $(\tau, \preccurlyeq)$ is contained in the quotient topology $\sigma$, as desired.
\end{proof}

In other words, the quasi-Polish quotient topology on $\uA(\Gamma, X, \mu)$ contains exactly the same information as the usual compact Polish topology together with the weak containment order.

\section{Topological versus measure-theoretic ergodic decompositions}
\label{sec:topmeasdecomp}

Let $E$ be a countable Borel equivalence relation on a standard Borel space $X$.  Recall (see e.g., \cite[I~\S2]{KM}) that a Borel measure $\mu$ on $X$ is \defn{$E$-invariant} if the following equivalent conditions hold:
\begin{itemize}
\item  $\mu$ is invariant with respect to some Borel action of a countable group $\Gamma \curvearrowright X$ inducing $E$;
\item  $\mu$ is invariant with respect to any Borel action of a countable group $\Gamma \curvearrowright X$ inducing $E$;
\item  for any two Borel sets $A, B \subseteq X$ such that there is a Borel bijection $f : A -> B$ with graph contained in $E$ (denoted $A \sim_E B$; see \cite[\S2]{DJK} or \cref{sec:kl}), we have $\mu(A) = \mu(B)$.
\end{itemize}
A Borel measure $\mu$ on $X$ is \defn{$E$-ergodic} if for any $E$-invariant Borel set $A \subseteq X$, we have $\mu(A) = 0$ or $\mu(X \setminus A) = 0$.  We say that $E$ is \defn{uniquely ergodic} if it admits a unique ergodic invariant probability Borel measure.  (Henceforth, by ``measure'' we mean Borel measure.)  Let $\@P(X)$ denote the standard Borel \defn{space of probability measures on $X$} (see \cite[\S17.E]{Kcdst}), $\INV_E \subseteq \@P(X)$ denote the subset of $E$-invariant measures, and $\EINV_E \subseteq \INV_E$ denote the subset of $E$-ergodic $E$-invariant measures.  It is well-known that $\INV_E, \EINV_E$ are Borel (see \cite[17.33]{Kcdst}, \cite[I~3.3]{KM}).

Recall (see e.g., \cite[\S2]{DJK}) that $E$ is \defn{compressible} if the following equivalent conditions hold:
\begin{itemize}
\item  there is a Borel injection $f : X -> X$ with graph contained in $E$ such that $X \setminus f(X)$ is an \defn{$E$-complete section} (i.e., $[X \setminus f(X)]_E = X$);
\item  $E \cong E \times I_\#N$, where $I_\#N$ is the indiscrete equivalence relation $\#N^2$ on $\#N$;
\item  there are no $E$-invariant probability measures (\defn{Nadkarni's theorem}; see \cite{Nad}, \cite[\S4.3]{BK}).
\end{itemize}
The \defn{uniform (measure-theoretic) ergodic decomposition theorem} of Farrell and Varadarajan states that there is a standard Borel decomposition of non-compressible $E$ into invariant, uniquely ergodic pieces (see e.g., \cite[I~3.3]{KM}, \cite[9.5]{DJK}):

\begin{theorem}[Farrell, Varadarajan]
\label{thm:measdecomp}
Let $E$ be a countable Borel equivalence relation on $X$.  Suppose $E$ is not compressible.  Then there is a Borel $E$-invariant surjection $p : X ->> \EINV_E$, such that
\begin{itemize}
\item[(i)]  for each $\mu \in \EINV_E$, $\mu|p^{-1}(\mu)$ is the unique $E|p^{-1}(\mu)$-ergodic invariant probability measure;
\item[(ii)]  for each $\mu \in \INV_E$, we have $\mu = \int p \,d\mu$.
\end{itemize}
Moreover, such $p$ (satisfying only (i)) is unique modulo compressible sets.
\end{theorem}

In this section, we show that this measure-theoretic ergodic decomposition may be realized in a particularly nice way: namely, as an instance of the \emph{topological} ergodic decomposition of \cref{thm:topdecomp}, for a suitably chosen Polish topology on $X$.  Furthermore, we may include in the decomposition not only the $E$-invariant \emph{probability} measures, but also all $E$-invariant $\sigma$-finite measures which are ``regular'' with respect to the topology, in the following weak sense.  We say that a $\sigma$-finite measure $\mu$ on a Polish space $X$ is \defn{totally singular} if for every open $U \subseteq X$, either $\mu(U) = 0$ or $\mu(U) = \infty$.  By a ``regular'' measure, we mean one that is not totally singular.

Let $(X, E)$ be a countable Borel equivalence relation, and fix a countable group $\Gamma$ with a Borel action on $X$ inducing $E$.  We say that a Polish topology on $X$ is \defn{good} if it generates the Borel structure on $X$ and makes the $\Gamma$-action continuous (hence makes $E$ satisfy the hypotheses of \cref{thm:topdecomp}).  In the following, by ``cofinally many'', we mean that any good Polish topology may be refined to one with the specified properties.

\begin{theorem}
\label{thm:topmeasdecomp}
For cofinally many good Polish topologies on $X$, the topological ergodic decomposition $p : X ->> X//E$ has the following properties:
\begin{enumerate}
\item[(i)]  $X//E$ is Polish;
\item[(ii)]  each component $C \in X//E$ admits, modulo scaling, at most one non-totally-singular $E|C$-invariant $\sigma$-finite measure, and this measure is ergodic (if it exists);
\item[(iii)]  the sets
\begin{align*}
R &:= \{C \in X//E \mid E|C \text{ admits a non-totally-singular invariant $\sigma$-finite measure}\}, \\
P &:= \{C \in X//E \mid E|C \text{ admits an invariant probability measure}\}
\end{align*}
are clopen in $X//E$;
\item[(iv)]  there is an open set $S \subseteq X$, such that $p(S) = R$, $S$ is a complete $E|p^{-1}(R)$-section, $p^{-1}(P) \subseteq S$, and there is a Borel isomorphism
\begin{align*}
R &\overset{\cong}{-->} \EINV_{E|S} \\
C &|--> \mu_C
\end{align*}
taking each component $C \in R$ admitting a non-totally-singular invariant (ergodic) $\sigma$-finite measure to the restriction of such a measure to $S \cap C$, with the resulting measure $\mu_C$ an ergodic probability measure such that $\mu_C(S \cap C) = 1$.
\end{enumerate}
\end{theorem}

\begin{remark}
Condition (iv) says, informally, that we may identify the set $R \subseteq X//E$ with the space of non-totally-singular ergodic invariant $\sigma$-finite measures modulo scaling, so that the projection map $p$ takes a point in a component $C$ supporting such a measure to the unique such measure, as in \cref{thm:measdecomp}.  However, the space of $\sigma$-finite measures does not have a natural standard Borel structure, so we have to represent such measures via their finite restrictions to some open set $S$.

By restricting the set $S$ and the map $C |-> \mu_C$ to $P \subseteq R$, we recover the usual ergodic decomposition (\cref{thm:measdecomp}): the composite
\begin{align*}
p^{-1}(P) --->{p} P --->{C |-> \mu_C} \EINV_{E|p^{-1}(P)} \cong \EINV_E
\end{align*}
has property (i) in \cref{thm:measdecomp}, hence (by uniqueness) may be identified with the map $p$ in \cref{thm:measdecomp}.  (The last isomorphism above follows from observing that $E|(X \setminus p^{-1}(P))$ is compressible.  To extend the above composite to a map defined on all of $X$ as in \cref{thm:measdecomp}, simply absorb this compressible set into any component $C \in p^{-1}(P)$.)
\end{remark}

We devote the rest of this section to the proof of \cref{thm:topmeasdecomp}, which essentially consists of repeatedly taking the usual ergodic decomposition (\cref{thm:measdecomp}) and refining the topology to make all of the desired properties hold.  As a way of organizing this iteration, we introduce the following notion: we say that a class $\@C$ of good Polish topologies on $X$ is a \defn{club} if it is cofinal in the above sense (i.e., any (good) Polish topology may be refined to one in $\@C$), as well as closed under countable increasing joins (i.e., if $\tau_0 \subseteq \tau_1 \subseteq \tau_2 \subseteq \dotsb \in \@C$, then the (good Polish) topology generated by $\bigcup_i \tau_i$ is in $\@C$).  Similar terminology is used in \cite[5.1.4]{BK}.


\begin{lemma}
\label{lm:club-intersect}
For countably many clubs $\@C_i$ of good Polish topologies on $X$, $\bigcap_i \@C_i$ is still a club.
\end{lemma}
\begin{proof}
Clearly $\bigcap_i \@C_i$ is closed under countable increasing joins.  To show that it is cofinal, let $f : \#N -> \#N$ be a surjection taking each value infinitely often, let $\tau_0$ be a good Polish topology on $X$, and recursively let $\tau_{i+1} \in \@C_{f(i)}$ refine $\tau_i$; then the join of the $\tau_i$ refines $\tau_0$ and is in each $\@C_i$.
\end{proof}

\begin{lemma}
\label{lm:club-polish}
The class of good Polish topologies $\tau$ on $X$ such that $(X, \tau)//E$ is Polish is a club.
\end{lemma}
\begin{proof}
By \cref{thm:topdecomp-polish} and \cref{thm:topdecomp-filtcolim}.
\end{proof}

We will need the following standard fact on extending invariant measures; see \cite[3.2]{DJK}.

\begin{proposition}
\label{thm:meas-extend}
Let $A \subseteq X$ be a Borel set and $\mu$ be an $E|A$-invariant $\sigma$-finite measure.  Then there is a unique $E$-invariant $\sigma$-finite measure $[\mu]_E$ such that $[\mu]_E|A = \mu$ and $[\mu]_E(X \setminus [A]_E) = 0$.

Explicitly, $[\mu]_E$ is given as follows: enumerate $\Gamma = \{\gamma_0, \gamma_1, \dotsc\}$, and let $B_i := (\gamma_i \cdot A) \setminus \bigcup_{j < i} (\gamma_j \cdot A)$.
Then
\begin{align*}
[\mu]_E(B) &:= \sum_i \mu(\gamma_i^{-1} \cdot (B \cap B_i)).
\qed
\end{align*}
\end{proposition}

For a Borel set $A \subseteq X$, we say that a good Polish topology $\tau$ on $X$ \defn{splits} $A$ if for every component $C \in (X, \tau)//E$, $E|(A \cap C)$ admits at most one invariant probability measure.

\begin{lemma}
\label{lm:club-split}
For every Borel $A \subseteq X$, all sufficiently fine good Polish topologies on $X$ split $A$.
\end{lemma}
\begin{proof}
If $E|A$ is compressible, then clearly any topology splits $A$.
Otherwise, let $q : A ->> \EINV_{E|A}$ be an ergodic decomposition as in \cref{thm:measdecomp}, and let $B_i \subseteq \EINV_{E|A}$ be a countable separating family of Borel sets.  Then any good Polish topology $\tau$ making each $[q^{-1}(B_i)]_E \subseteq X$ clopen splits $A$.  Indeed, given such $\tau$, for any $C \in (X, \tau)//E$, if there is some $x \in A \cap C$, then putting
\begin{align*}
F := \bigcap_{B_i \ni q(x)} [q^{-1}(B_i)]_E \cap \bigcap_{B_i \not\ni q(x)} (X \setminus [q^{-1}(B_i)]_E),
\end{align*}
$F$ is $\tau$-closed and $E$-invariant and contains $x$, hence contains its component $\den{x}_E^\tau = C$; but clearly $q^{-1}(q(x)) = A \cap F$, whence $A \cap C \subseteq A \cap F = q^{-1}(q(x))$, whence $E|(A \cap C)$ has at most one invariant probability measure (namely, $q(x)$) by definition of $q$.
\end{proof}

We say that a good Polish topology $\tau$ on $X$ is \defn{very good} if $\tau$ splits every $\tau$-open set.

\begin{lemma}
\label{lm:split-basis}
If a good Polish topology $\tau$ on $X$ splits every set in a basis for $\tau$, then $\tau$ is very good.
\end{lemma}
\begin{proof}
For any component $C \in (X, \tau)//E$, since $E|C$ is minimal, every $\tau$-open $U \subseteq X$ which intersects $C$ intersects every equivalence class in $C$.  So if for some $C \in (X, \tau)//E$ and $\tau$-open $U \subseteq X$, $E|(U \cap C)$ had two distinct invariant probability measures $\mu, \nu$, then letting $V \subseteq U$ be basic open with $V \cap C \ne \emptyset$, we have that $\mu|(V \cap C), \nu|(V \cap C)$ are nonzero finite $E|(V \cap C)$-invariant measures, which must be distinct modulo scaling by \cref{thm:meas-extend} (since they extend to $\mu, \nu$ respectively), whence $E|(V \cap C)$ also has two distinct invariant probability measures.
\end{proof}

\begin{lemma}
\label{lm:club-split-open}
The class of very good Polish topologies on $X$ is a club.
\end{lemma}
\begin{proof}
Closure under countable increasing join follows from \cref{lm:split-basis}.  To check cofinality: given any good $\tau_0$, repeatedly apply \cref{lm:club-split} to obtain $\tau_1 \supseteq \tau_0$ which splits all sets in a countable basis for $\tau_0$, then similarly obtain $\tau_2 \supseteq \tau_1$, etc.; the join $\tau$ of $\tau_0 \subseteq \tau_1 \subseteq \tau_2 \subseteq \dotsb$ then splits every $\tau$-open set, by \cref{lm:split-basis}.
\end{proof}

\begin{lemma}
\label{lm:split-nontotsing}
Every very good Polish topology $\tau$ on $X$ satisfies condition (ii) in \cref{thm:topmeasdecomp}.
\end{lemma}
\begin{proof}
Suppose $C \in (X, \tau)//E$ has two non-totally-singular $E|C$-invariant $\sigma$-finite measures $\mu, \nu$.  Then $\mu(U \cap C), \nu(V \cap C) \in (0, \infty)$ for some $\tau$-open $U, V \subseteq X$.  Since $E|C$ is minimal, $U, V$ intersect every $E|C$-class.  So there is a group element $\gamma \in \Gamma$ such that $(\gamma \cdot U) \cap V \cap C \ne \emptyset$.  Then $W := (\gamma \cdot U) \cap V$ still intersects every $E|C$-class, whence $\mu(W \cap C), \nu(W \cap C) \in (0, \infty)$.  Since $\tau$ splits $W$, $\mu|(W \cap C) = r \nu|(W \cap C)$ for some $r \in (0, \infty)$; since $W$ intersects every $E|C$-class, by \cref{thm:meas-extend} we have $\mu = [\mu|(W \cap C)]_{E|C} = r[\nu|(W \cap C)]_{E|C} = r\nu$.

To check that a non-totally-singular $E|C$-invariant measure $\mu$ is necessarily ergodic, let $U \subseteq X$ be $\tau$-open so that $\mu(U \cap C) \in (0, \infty)$; since $\tau$ splits $U$, $\mu|(U \cap C)$ is $E|(U \cap C)$-ergodic, whence since $U$ intersects every $E|C$-class, $\mu$ is $E|C$-ergodic.
\end{proof}

For a very good Polish topology $\tau$ on $X$ (thus \cref{thm:topmeasdecomp}(ii) holds by \cref{lm:split-nontotsing}), let
\begin{align*}
R(\tau), P(\tau) \subseteq (X, \tau)//E
\end{align*}
denote the sets $R, P$ defined in \cref{thm:topmeasdecomp}(iii).  For a $\tau$-open $U \subseteq X$, let
\begin{align*}
P(U, \tau) := \{C \in (X, \tau)//E \mid \INV_{E|(U \cap C)} \ne \emptyset\}.
\end{align*}
Clearly, $P(X, \tau) = P(\tau)$.

\begin{lemma}
\label{lm:split-open-rpu}
For any basis $\@U$ for $\tau$, $R(\tau) = \bigcup_{U \in \@U} P(U, \tau)$.
\end{lemma}
\begin{proof}
$\subseteq$ is because for $C \in (X, \tau)//E$, any $E|C$-invariant non-totally-singular measure $\mu$ restricts to an $E|(U \cap C)$-invariant finite measure for some $U \in \@U$; $\supseteq$ is because any $E|(U \cap C)$-invariant probability measure extends to an $E|C$-invariant non-totally-singular measure (\cref{thm:meas-extend}).
\end{proof}

\begin{lemma}
\label{lm:split-open-rp-borel}
The sets $R(\tau), P(\tau), P(U, \tau) \subseteq (X, \tau)//E$ above are Borel.
\end{lemma}
\begin{proof}
By the above, it is enough to check that $P(U, \tau) \subseteq (X, \tau)//E$ is Borel.  Indeed, it is the preimage, under the embedding
\begin{align*}
(X, \tau)//E &`-> \@P((X, \tau)//E)
\end{align*}
taking $C \in (X, \tau)//E$ to the Dirac delta $\delta_C$, of the image of the measure pushforward map
\begin{align*}
p_* : \EINV_{E|U} --> \@P((X, \tau)//E)
\end{align*}
(where $p : U \subseteq X ->> (X, \tau)//E$ is the projection), which is injective because $\tau$ splits $U$.
\end{proof}

\begin{lemma}
\label{lm:split-open-rp-filtcolim}
Let $\tau_0 \subseteq \tau_1 \subseteq \dotsb$ be a sequence of very good Polish topologies on $X$, and let $\tau$ be their join (which is very good by \cref{lm:club-split-open}).  Let
\begin{align*}
p_i : (X, \tau_i) &->> (X, \tau_i)//E, &
p : (X, \tau) &->> (X, \tau)//E
\end{align*}
denote the quotient projections.  Then
\begin{align*}
p^{-1}(R(\tau)) &= \bigcap_i p_i^{-1}(R(\tau_i)), &
p^{-1}(P(\tau)) &= \bigcap_i p_i^{-1}(P(\tau_i)),
\end{align*}
and for $\tau_0$-open $U \subseteq X$,
\begin{align*}
p^{-1}(P(U, \tau)) &= \bigcap_i p_i^{-1}(P(U, \tau_i)).
\end{align*}
\end{lemma}
\begin{proof}
We check the last case; the other two are similar.
For $x \in X$, we have $x \in p^{-1}(P(U, \tau))$ iff $E|(U \cap \den{x}_E^\tau)$ admits an invariant probability measure, while $x \in \bigcap_i p_i^{-1}(P(U, \tau_i))$ iff for every $i$, $E|(U \cap \den{x}_E^{\tau_i})$ admits an invariant probability measure.
The former clearly implies the latter.  Conversely, if the latter holds, let $\mu_i$ be the measure on $E|(U \cap \den{x}_E^{\tau_i})$.  Since each $\tau_i$ splits $U$, each $\mu_i$ is the unique $E|(U \cap \den{x}_E^{\tau_i})$-invariant probability measure, whence in fact all the $\mu_i$ are identical and supported on $U \cap \bigcap_i \den{x}_E^{\tau_i}$.  By \cref{thm:topdecomp-filtcolim}, $\bigcap_i \den{x}_E^{\tau_i} = \den{x}_E^\tau$, whence $E|(U \cap \den{x}_E^\tau)$ admits an invariant probability measure.
\end{proof}

\begin{lemma}
\label{lm:club-split-open-rp-closed}
The following class of good Polish topologies $\tau$ on $X$ forms a club: $\tau$ is very good, the sets $R(\tau), P(\tau) \subseteq (X, \tau)//E$ are closed, and there is a countable basis $\@U$ for $\tau$ such that $P(U, \tau) \subseteq (X, \tau)//E$ is closed for every $U \in \@U$.
\end{lemma}
\begin{proof}
Closure under countable increasing join follows from \cref{lm:club-split-open} and \cref{lm:split-open-rp-filtcolim}.
To check cofinality: let $\tau_0$ be any good Polish topology on $X$; given $\tau_{2n}$, let $\tau_{2n+1} \supseteq \tau_{2n}$ be a very good Polish topology by \cref{lm:club-split-open}, and let $\tau_{2n+2} \supseteq \tau_{2n+1}$ be a good Polish topology in which the preimages under the projection
\begin{align*}
p_{2n+1} : X ->> (X, \tau_{2n+1})//E
\end{align*}
of the Borel (by \cref{lm:split-open-rp-borel}) sets $R(\tau_{2n+1}), P(\tau_{2n+1}), P(U, \tau_{2n+1}) \subseteq (X, \tau_{2n+1})//E$ are clopen, for all $U$ in some countable basis $\@U_{2n+1}$ for $\tau_{2n+1}$.  Let $\tau$ be the join of the $\tau_i$ and $p : (X, \tau) ->> (X, \tau)//E$ be the projection.  Then by \cref{lm:club-split-open}, $\tau$ is very good, while by \cref{lm:split-open-rp-filtcolim}, the sets $p^{-1}(R(\tau)), p^{-1}(P(\tau)), p^{-1}(P(U, \tau)) \subseteq X$ are $\tau$-closed, i.e., $R(\tau), P(\tau), P(U, \tau) \subseteq (X, \tau)//E$ are closed, for all $U$ in the countable basis $\@U := \bigcup_n \@U_{2n+1}$ for $\tau$.
\end{proof}

\begin{proof}[Proof of \cref{thm:topmeasdecomp}]
Let $\tau_0$ be a good Polish topology on $X$, and let $\tau_1 \supseteq \tau_0$ be a finer topology satisfying \cref{lm:club-polish} and \cref{lm:club-split-open-rp-closed}, with the latter giving a basis $\@U$.  So (i--ii) hold for the decomposition $p : (X, \tau_1) ->> (X, \tau_1)//E$.  Let $\tau_2 \supseteq \tau_1$ be given by adjoining the $E$-invariant $\tau_1$-closed sets $p^{-1}(R(\tau_1)), p^{-1}(P(\tau_1)), p^{-1}(P(U, \tau_1)) \subseteq X$, for all $U \in \@U$; by \cref{thm:topdecomp-closed}, doing so does not change the components of the topological ergodic decomposition, i.e., $(X, \tau_2)//E = (X, \tau_1)//E$ as sets.  So (ii) continues to hold for $\tau_2$; clearly so does (i), and also (iii) holds since $R(\tau_2), P(\tau_2) \subseteq (X, \tau_2)//E$ (which are the same sets as $R(\tau_1), P(\tau_1)$) are now clopen.

Finally, we check (iv).  Let $\@U = \{U_0, U_1, \dotsc\}$, and put
\begin{align*}
V_i &:= (U_i \cap p^{-1}(P(U_i, \tau_1))) \setminus (p^{-1}(P) \cup \bigcup_{j < i} p^{-1}(P(U_j, \tau_1))).
\end{align*}
Since $P$ (i.e., $P(\tau_1) = P(\tau_2)$) and $P(U_i, \tau_1)$ are $\tau_2$-clopen, so is each $V_i$.  It is easily seen that
\begin{align*}
R = P \sqcup \bigsqcup_i p(V_i)
\end{align*}
(using \cref{lm:split-open-rpu}), with $p(V_i) = P(U_i, \tau_1) \setminus (P \cup \bigcup_{j < i} P(U_j, \tau_1))$.  Put
\begin{align*}
S := p^{-1}(P) \cup \bigcup_i V_i.
\end{align*}
Clearly $p(S) = R$ and $p^{-1}(P) \subseteq S$.  The map $C |-> \mu_C \in \EINV_{E|S}$ is defined in the obvious way: $\mu_C$ is the unique such measure so that $p_*(\mu_C) = \delta_C \in \@P((X, \tau_2)//E)$.  For $C \in P$, $\mu_C$ exists by definition of $P$; similarly, for $C \in p(V_i)$, $\mu_C$ exists by definition of $P(U_i, \tau_1) \supseteq p(V_i)$.
\end{proof}

\begin{remark}
If in the above proof we do not refine $\tau_1$ to $\tau_2$, then we obtain the following variant of the statement of \cref{thm:topmeasdecomp}: there is a club of topologies satisfying the conditions (and not just cofinally many); but the sets $R, P$ in (iii) are merely closed (instead of clopen), and the set $S$ in (iv) will only be $F_\sigma$ with open sections (instead of open).
\end{remark}

\section{Cardinal algebras of equidecomposition types}
\label{sec:kl}

\subsection{Cardinal algebras}
\label{sec:ca}

A \defn{cardinal algebra} is an algebraic structure $(A, 0, +, \sum)$, where $(A, 0, +)$ is an abelian monoid and $\sum : A^\#N -> A$ is a countably infinitary operation on $A$ with $\sum(a_i)_{i \in \#N}$ denoted also by $\sum_{i < \infty} a_i$, satisfying the following axioms:
\begin{enumerate}[label=(\Alph*)]

\item \label{ca:assoc}
$\sum_{i < \infty} a_i = a_0 + \sum_{i < \infty} a_{i+1}$.

\item \label{ca:dist}
$\sum_{i < \infty} (a_i + b_i) = \sum_{i < \infty} a_i + \sum_{i < \infty} b_i$.

\item \label{ca:refinement}
If $a + b = \sum_{i < \infty} c_i$, then there are $(a_i)_{i < \infty}, (b_i)_{i < \infty}$ such that $a = \sum_{i < \infty} a_i$, $b = \sum_{i < \infty} b_i$, and $a_i + b_i = c_i$.
This is depicted in the following picture:
\begin{equation*}
\begin{array}{c|cccc}
& c_0 & c_1 & c_2 & \dotsb \\
\hline
a & a_0 & a_1 & a_2 & \dotsb \\
b & b_0 & b_1 & b_2 & \dotsb
\end{array}
\end{equation*}

\item \label{ca:remainder}
If $(a_i)_{i < \infty}, (b_i)_{i < \infty}$ are such that $a_i = b_i + a_{i+1}$, then there is $c$ such that $a_i = c + \sum_{j < \infty} b_{i+j}$.

\end{enumerate}
Cardinal algebras were introduced and comprehensively studied by Tarski \cite{Tar}; the above axioms are from \cite{KMd} and are equivalent to Tarski's original axioms.  These axioms imply many other desirable algebraic properties, of which the following will be most important for our purposes:
\begin{enumerate}[resume*]

\item \label{ca:addition}
\cite[1.17, 1.38, 1.42]{Tar}  Addition is well-behaved: for finitely many elements $a_0, a_1, \dotsc, a_{n-1} \in A$, we may define their sum via the equivalent formulas
\begin{align*}
\sum_{i < n} a_i := a_0 + \dotsb + a_{n-1} = \sum (a_0, \dotsc, a_{n-1}, 0, 0, \dotsc);
\end{align*}
and both finitary and infinitary addition satisfy all commutativity and associativity laws.

\item \label{ca:order}
\cite[1.31, 1.22]{Tar}  We have a canonical partial order, defined by
\begin{align*}
a \le b \iff \exists c\, (a + c = b),
\end{align*}
which interacts well with addition: $0 \le a$ for all $a$; and if $a_i \le b_i$ for each $i$, then $\sum_i a_i \le \sum_i b_i$.

\item \label{ca:sfiltjoin}
\cite[2.24, 2.21, 3.19]{Tar}  Countable increasing joins exist: given $a_0 \le a_1 \le \dotsb \in A$, there is a join (i.e., least upper bound) $\bigvee_i a_i \in A$.  Moreover, for all $a_0, a_1, \dotsc \in A$ we have
\begin{align*}
\sum_{i < \infty} a_i = \bigvee_{n < \infty} \sum_{i < n} a_i.
\end{align*}

\item \label{ca:meet}
\cite[3.4]{Tar}  If two elements $a, b \in A$ have a meet $a \wedge b \in A$, then they also have a join $a \vee b \in A$, satisfying
\begin{align*}
a + b = a \wedge b + a \vee b.
\end{align*}

\item \label{ca:nmult}
\cite[1.43, 1.45]{Tar}  For $n \le \infty$, put
\begin{align*}
n \cdot a := \sum_{i < n} a.
\end{align*}
This yields an action of the multiplicative monoid $\-{\#N} := \#N \cup \{\infty\}$ (where $0 \infty := 0$) on $A$, which preserves the partial order and countable addition in both $\-{\#N}$ and $A$.

\item \label{ca:div}
\cite[2.34]{Tar}  For $0 < n < \infty$, we say that $a \in A$ is \defn{divisible by $n$} if there is a $b$ such that $n \cdot b = a$; such $b$ is necessarily unique, hence may be denoted by $a/n$.  We say that $a$ is \defn{completely divisible} if it is divisible by arbitrarily large $n$.

\item \label{ca:rmult}
\cite[end of \S2]{Tar} \cite[1.1--1.13]{Ch1}  For completely divisible $a \in A$, we may define \defn{real multiples} $r \cdot a$ for every $r \in \-{\#R^+} := [0, \infty]$ by
\begin{align*}
r \cdot a := \sum_i (p_i \cdot a/q_i)
\end{align*}
for any sequence of rationals $p_i/q_i$ with sum $r$ such that $a$ is divisible by each $q_i$; the definition does not depend on the choice of such sequence.  This yields an action of the multiplicative monoid $\-{\#R^+}$ (where $0\infty := 0$) on $A$, extending the action of $\-{\#N} \subseteq \-{\#R^+}$, which preserves the partial order and countable addition in both $\-{\#R^+}$ and $A$.

\end{enumerate}

\subsection{The algebra $\@K(E)$}

Fix a compressible countable Borel equivalence relation $(X, E)$.  Let $\@B(X)$ denote the Borel $\sigma$-algebra of $X$.  Recall (see e.g., \cite[\S2]{DJK}) that for $A, B \in \@B(X)$, an \defn{$E$-equidecomposition}
\begin{align*}
f : A \sim_E B
\end{align*}
is a Borel bijection $f : A -> B$ with graph contained in $E$; $A, B$ are \defn{$E$-equidecomposable}, written $A \sim_E B$, if there is some $f : A \sim_E B$.  We also write
\begin{align*}
A \preceq_E B, &&
A \prec_E B
\end{align*}
to mean respectively that $A \sim_E C$ for some Borel $C \subseteq B$, and that  such $C$ may be chosen so that $[B \setminus C]_E = [B]_E$.
(Note that it is possible to have $A \preceq_E B$ but also $A \succ_E B$, e.g., $X \prec_E X$, since $E$ is compressible.)

Put
\begin{align*}
\@K(E) := \@B(X)/{\sim_E}.
\end{align*}
The rest of this paper is devoted to the study of the algebraic structure of $\@K(E)$ (and the related $\@L(E)$ to be defined in the next section).  We will use the following notation: for $A \in \@B(X)$, write
\begin{align*}
\~A := [A]_{\sim_E}.
\end{align*}

We define finite and countably infinite sums in $\@K(E)$ as follows.  For countably many elements $\~A_0, \~A_1, \dotsc \in \@K(E)$, by compressibility of $E$, we may choose the representatives $A_0, A_1, \dotsc \in \@B(X)$ to be pairwise disjoint; put
\begin{align*}
\sum_i \~A_i := \~{\bigcup_i A_i} \qquad\text{for pairwise disjoint $A_0, A_1, \dotsc$}.
\end{align*}
It is straightforward that this is well-defined (given $f_i : A_i \sim_E B_i$ where the $B_i$ are also pairwise disjoint, we have $\bigcup_i f_i : \bigcup_i A_i \sim_E \bigcup_i B_i$).  Put also
\begin{align*}
0 := \~\emptyset, &&
\infty := \~X.
\end{align*}

\begin{proposition}
\label{thm:k-ca}
$\@K(E)$ is a cardinal algebra, with addition as above and canonical partial order given by
\begin{align*}
\~A \le \~B \iff A \preceq_E B.
\end{align*}
\end{proposition}
\begin{proof}
Recall that for any countable group $\Gamma$ with a Borel action $\Gamma \curvearrowright X$ inducing $E$, we have $A \sim_E B$ iff there are Borel partitions $A = \bigsqcup_{\gamma \in \Gamma} A_\gamma$ and $B = \bigsqcup_{\gamma \in \Gamma} B_\gamma$ such that $\gamma \cdot A_\gamma = B_\gamma$; see e.g., \cite[\S4.2--3]{BK}.  Thus $\@K(E)$ is an instance of the cardinal algebra of equidecomposition types constructed in \cite[16.7]{Tar} (see also \cite[2.4]{Ch1}).

(It is also easy to verify axioms (A--D) from \cref{sec:ca} directly, by picking the Borel sets involved in each axiom to be pairwise disjoint, using compressibility of $E$.)

That $\~A \le \~B \iff A \preceq_E B$ is immediate from the definitions.
\end{proof}

A key tool in analyzing the structure of $\@K(E)$ is the following lemma, first used by Becker--Kechris \cite[4.5.1]{BK} in their proof of the general case of Nadkarni's theorem:

\begin{lemma}[Becker--Kechris]
\label{lm:compare}
For any $A, B \in \@B(X)$, there is an $E$-invariant Borel partition $X = Y \sqcup Z$ such that $A \cap Y \preceq_E B \cap Y$ and $A \cap Z \succ_E B \cap Z$.
\end{lemma}

\begin{proposition}
\label{thm:k-meetjoin}
$\@K(E)$ has finite meets, hence also countable joins.
\end{proposition}
\begin{proof}
Clearly the greatest element is $\infty \in \@K(E)$.  To compute the meet of $\~A, \~B \in \@K(E)$, let $Y, Z$ be given by \cref{lm:compare}; then it is easily seen that
\begin{align*}
\~A \wedge \~B = [(A \cap Y) \cup (B \cap Z)]_{\sim_E}.
\end{align*}
By \S\ref{sec:ca}\ref{ca:meet}, it follows that $\@K(E)$ has binary joins, hence (since every cardinal algebra has least element $0$ and countable increasing joins by \S\ref{sec:ca}\ref{ca:sfiltjoin}) arbitrary countable joins.

Alternatively, we may compute joins directly, as follows.  Similarly to meets, for $Y, Z$ as above,
\begin{align*}
\~A \vee \~B = [(A \cap Z) \cup (B \cap Y)]_{\sim_E}.
\end{align*}
To compute the increasing join of $\~A_0 \le \~A_1 \le \dotsb \in \@K(E)$, using compressibility of $E$, we may choose the representatives $A_0, A_1, \dotsc \in \@B(X)$ so that $A_0 \subseteq A_1 \subseteq \dotsb$; then
\begin{align*}
\bigvee_i \~A_i &= [\bigcup_i A_i]_{\sim_E} \qquad\text{for $A_0 \subseteq A_1 \subseteq \dotsb$}.
\end{align*}
To check that this works, use \S\ref{sec:ca}\ref{ca:sfiltjoin}: we have $\bigcup_i A_i = A_0 \sqcup \bigsqcup_i (A_{i+1} \setminus A_i)$, whence $[\bigcup_i A_i]_{\sim_E} = \~A_0 + \sum_i [A_{i+1} \setminus A_i]_{\sim_E} = \bigvee_n (\~A_0 + [A_1 \setminus A_0]_{\sim_E} + \dotsb + [A_n \setminus A_{n-1}]_{\sim_E}) = \bigvee_n \~A_n$.
Finally, we may compute an arbitrary countable join of $\~A_0, \~A_1, \dotsc \in \@K(E)$ as the increasing join of finite joins $\bigvee_n (\~A_0 \vee \dotsb \vee \~A_n)$.
\end{proof}

We next consider divisibility in $\@K(E)$, for which we use the following lemma \cite[7.4]{KM}:

\begin{lemma}[Kechris--Miller]
\label{lm:aperiodic-fse}
For every aperiodic countable Borel equivalence relation $(X, E)$ and $n > 0$, there is a finite Borel subequivalence relation $F \subseteq E$ all of whose classes have size $n$.
\end{lemma}

\begin{proposition}
\label{thm:k-div}
$\~A \in \@K(E)$ is completely divisible iff $E|A$ is aperiodic.
\end{proposition}
\begin{proof}
If $E|A$ is aperiodic, then for any $n > 0$, by \cref{lm:aperiodic-fse}, we may find a Borel subequivalence relation $F \subseteq E|A$ all of whose classes have size $n$; letting $A_0, \dotsc, A_{n-1} \subseteq A$ be disjoint Borel transversals of $F$, we clearly have $\~A = \~A_0 + \dotsb + \~A_{n-1}$ and $A_0 \sim_E \dotsb \sim_E A_{n-1}$, whence $\~A$ is divisible by $n$.  Conversely, if $E|A$ has a finite class, say of cardinality $n$, then clearly $\~A$ is not divisible by any $m > n$.
\end{proof}

We say that $\~A \in \@K(E)$ is \defn{finite} if $E|A$ is finite (i.e., has finite classes), and \defn{aperiodic} if $E|A$ is aperiodic, or equivalently if $\~A$ is completely divisible by \cref{thm:k-div}.  We let $\@K^\fin(E), \@K^\ap(E) \subseteq \@K(E)$ denote the subsets of finite, respectively aperiodic, elements.

\begin{proposition}
$\@K^\ap(E) \subseteq \@K(E)$ is a cardinal subalgebra.
\end{proposition}
\begin{proof}
Clearly $\@K^\ap(E) \subseteq \@K(E)$ is closed under countable addition; so it suffices to check that the existential axioms (C) and (D) from \S\ref{sec:ca} still hold in $\@K^\ap(E)$.  For (C), let $\~A, \~B, \~C_i \in \@K^\ap(E)$ with $\~A + \~B = \sum_i \~C_i$, and let $\~A_i, \~B_i \in \@K(E)$ with $\~A = \sum_i \~A_i$ and $\~B = \sum_i \~B_i$ be given by (C) in the cardinal algebra $\@K(E)$.  Let $Y \subseteq X$ be the union of all $E$-classes whose intersection with some $A_i$ or $B_i$ is finite nonempty.  Clearly $Y$ is $E$-invariant Borel, whence by restricting $E$, we may assume either $Y = \emptyset$ or $Y = X$.  If $Y = \emptyset$, we have $\~A_i, \~B_i \in \@K^\ap(E)$, so (C) holds.  If $Y = X$, then clearly $E$ is smooth, whence we may easily find $\~A_i, \~B_i$ making (C) hold.  The proof of (D) is similar.
\end{proof}

Let $\-{\#R^+} := [0, \infty]$, which is also a cardinal algebra with finite meets (and countable joins).  We say that a map $f : A -> B$ between cardinal algebras $A, B$ is a \defn{$\sum$-homomorphism} if it preserves countable sums (including zero).  Clearly, a $\sum$-homomorphism $\@K(E) -> \-{\#R^+}$, i.e., a $\sim_E$-invariant $\sigma$-additive map $\@B(X) -> \-{\#R^+}$, is the same thing as an $E$-invariant measure:

\begin{proposition}
\label{thm:k-homom-measure}
We have a canonical bijection
\begin{align*}
\INV^*_E &\cong \{\text{$\sum$-homomorphisms } \@K(E) -> \-{\#R^+}\} \\
\mu &|--> (\~A |-> \mu(A)),
\end{align*}
where $\INV^*_E$ denotes the set of (not necessarily probability or even $\sigma$-finite, and possibly zero) $E$-invariant measures on $X$.  \qed
\end{proposition}

We henceforth identify a measure $\mu \in \INV^*_E$ with the corresponding $\sum$-homomorphism.

\begin{lemma}
\label{lm:ca-homom-sfiltjoin-rmult}
Every $\sum$-homomorphism $f : A -> B$ between cardinal algebras preserves countable increasing joins as well as real multiples of completely divisible elements.
\end{lemma}
\begin{proof}
Let $a_0 \le a_1 \le \dotsb \in A$.  Then $a_{i+1} = a_i + b_i$ for some $b_i \in A$.  Using \S\ref{sec:ca}\ref{ca:sfiltjoin}, we have $\bigvee_i a_i = \bigvee_i (a_0 + b_0 + \dotsb + b_{i-1}) = a_0 + \sum_i b_i$, whence $f(\bigvee_i a_i) = f(a_0) + \sum_i f(b_i) = \bigvee_i (f(a_0) + f(b_0) + \dotsb + f(b_{i-1})) = \bigvee_i f(a_0 + b_0 + \dotsb + b_{i-1}) = \bigvee_i f(a_i)$.

Let $a \in A$ be completely divisible and $r \in \-{\#R^+}$.  Then for every positive integer $n$, we have $n \cdot f(a/n) = f(n \cdot a/n) = f(a)$, whence $f(a/n) = f(a)/n$.  So for any sequence of rationals $p_i/q_i$ with sum $r$, we have $f(r \cdot a) = f(\sum_i p_i \cdot a/q_i) = \sum_i p_i \cdot f(a)/q_i = r \cdot f(a)$.
\end{proof}

\begin{lemma}
\label{lm:ca-homom-meetjoin}
If a $\sum$-homomorphism $\mu : A -> \-{\#R^+}$ preserves binary meets, then it preserves countable joins.
\end{lemma}
\begin{proof}
For $a, b \in A$, from $a + b = a \wedge b + a \vee b$ (\S\ref{sec:ca}\ref{ca:meet}), we get $\mu(a \wedge b) + \mu(a \vee b) = \mu(a) + \mu(b) = \mu(a) \wedge \mu(b) + \mu(a) \vee \mu(b) = \mu(a \wedge b) + \mu(a) \vee \mu(b)$.  If $\mu(a \wedge b) < \infty$, then we may cancel to get $\mu(a \vee b) = \mu(a) \vee \mu(b)$.  Otherwise, since $a \wedge b \le a, b, a \vee b$, we have $\mu(a \vee b) = \infty = \mu(a) \vee \mu(b)$.  So $\mu$ preserves binary joins.  Since $\mu$ always preserves $0$ and countable increasing joins (\cref{lm:ca-homom-sfiltjoin-rmult}), it preserves arbitrary countable joins.
\end{proof}

\begin{lemma}
\label{lm:k-meas-ergodic}
$\mu \in \INV^*_E$ is ergodic iff $\mu : \@K(E) -> \-{\#R^+}$ preserves finite meets.
\end{lemma}
\begin{proof}
By compressibility of $E$, every nonzero $E$-invariant measure is infinite; thus $\mu$ preserves the greatest element $\infty$ iff $\mu$ is nonzero.  If $\mu$ is nonzero and preserves binary meets, then for every $E$-invariant Borel $A \subseteq X$, we have $\~A \wedge \~{X \setminus A} = 0$, whence $\mu(A) \wedge \mu(X \setminus A) = \mu(\~A \wedge \~{X \setminus A}) = \mu(0) = 0$, whence either $\mu(A) = 0$ or $\mu(X \setminus A) = 0$, i.e., $\mu$ is ergodic.  Conversely, if $\mu$ is ergodic, then for every $\~A, \~B \in \@K(E)$, letting $X = Y \sqcup Z$ be given by \cref{lm:compare}, we have either $\mu(Y) = 0$ or $\mu(Z) = 0$; in the former case, we have $\mu(A) = \mu(A \cap Z) \ge \mu(B \cap Z) = \mu(B)$ (since $A \cap Z \succ_E B \cap Z$), whence $\mu(\~A \wedge \~B) = \mu((A \cap Y) \cup (B \cap Z)) = \mu(B \cap Z) = \mu(B) = \mu(A) \wedge \mu(B)$, while in the latter case we similarly have $\mu(\~A \wedge \~B) = \mu(A) = \mu(A) \wedge \mu(B)$.
\end{proof}

For a cardinal algebra $A$, we say that a map $\mu : A -> \-{\#R^+}$ is a \defn{$(\sum, \wedge, \bigvee, \-{\#R^+})$-homomorphism} if it preserves countable sums, finite meets, countable joins, and real multiples of completely divisible elements.  The preceding lemmas now give

\begin{proposition}
We have a canonical bijection
\begin{align*}
\EINV^*_E &\cong \{\text{$(\sum, \wedge, \bigvee, \-{\#R^+})$-homomorphisms } \@K(E) -> \-{\#R^+}\}
\end{align*}
where $\EINV^*_E \subseteq \INV^*_E$ denotes the (not necessarily $\sigma$-finite) $E$-ergodic invariant measures.  \qed
\end{proposition}

\begin{remark}
Non-$\sigma$-finite measures are not so tractable: for any $\sigma$-complete ultrafilter $\@U$ of $E$-invariant Borel subsets of $X$, we have an $E$-ergodic invariant measure $\mu \in \EINV^*_E$, given by $\mu(A) = 0$ if $[A]_E \not\in \@U$, else $\mu(A) = \infty$.
\end{remark}

Finally in this section, we show that there are ``enough'' homomorphisms $\@K(E) -> \-{\#R^+}$.  Because of the preceding remark, we will in fact only consider homomorphisms corresponding to $\sigma$-finite measures.  Let $\EINV^\sigma_E \subseteq \EINV^*_E$ denote the subset of $\sigma$-finite measures.

\begin{lemma}
\label{lm:k-embed}
For any $\~A \not\le \~B \in \@K(E)$, there is a $\mu \in \EINV^\sigma_E$ such that $\mu(\~A) > \mu(\~B)$.
\end{lemma}
\begin{proof}
Since $A \not\preceq_E B$, by restricting $E$ to the set $Z$ given by \cref{lm:compare}, we may assume that $A \succ_E B$.  If $[A]_E \not\subseteq [B]_E$, then we may let $\mu$ be an atomic measure; so we may restrict $E$ to $[A]_E$, and assume both $A$ and $B$ are $E$-complete sections.  If $E|B$ were compressible, then we would have $\~B = \infty \ge \~A$, a contradiction.  Thus $E|B$ is not compressible, hence has an ergodic invariant probability measure $\mu$ by Nadkarni's theorem.  Extending $\mu$ to $E$ using \cref{thm:meas-extend}, we have $\mu(A) > \mu(B)$ since $A \succ_E B$ and $\mu(B) < \infty$, as desired.
\end{proof}

\begin{proposition}
\label{thm:k-embed}
We have an embedding
\begin{align*}
\eta : \@K(E) &`-> \-{\#R^+}^{\EINV^\sigma_E} \\
\~A &|-> (\mu |-> \mu(A))
\end{align*}
preserving countable sums, finite meets, countable joins, and real multiples of completely divisible elements (with the pointwise operations in $\-{\#R^+}^{\EINV^\sigma_E}$).  \qed
\end{proposition}

By universal algebra, we may rephrase this result as follows.  A \defn{Horn axiom} in the operations $\sum, \wedge, \bigvee, \-{\#R^+}$ (countable sums, finite meets, countable joins, and real multiples of completely divisible elements) is an axiom of the form
\begin{align*}
\forall \vec{v}\, [\bigwedge_i (s_i(\vec{v}) = t_i(\vec{v})) -> (s(\vec{v}) = t(\vec{v}))]
\end{align*}
where $s_i, t_i, s, t$ are terms built from the specified operations and the (possibly infinitely many) variables $\vec{v}$ (in the case of the partially defined operations $\-{\#R^+}$, we interpret the right-hand side of the implication to mean ``if both terms are defined, then the equality holds'').

\begin{corollary}
$\@K(E)$ obeys all Horn axioms in the operations $\sum, \wedge, \bigvee, \-{\#R^+}$ which hold in the algebra $\-{\#R^+}$.  \qed
\end{corollary}

We end this section by summarizing all the properties of the algebra $\@K(E)$ we have considered:

\begin{theorem}
\label{thm:k}
$\@K(E)$ is a cardinal algebra, with finite meets and (hence) countable joins, and with completely divisible elements coinciding with the aperiodic ones $\@K^\ap(E) \subseteq \@K(E)$ which form a cardinal subalgebra.
We have canonical bijections (where $\bigvee^\up$ denotes countable increasing joins)
\begin{align*}
\INV^*_E &\cong \{\text{$\sum$-homomorphisms } \@K(E) -> \-{\#R^+}\} \\
&= \{\text{$(\sum, \bigvee^\up, \-{\#R^+})$-homomorphisms } \@K(E) -> \-{\#R^+}\}, \\
\EINV^*_E &\cong \{\text{$(\sum, \wedge, \bigvee, \-{\#R^+})$-homomorphisms } \@K(E) -> \-{\#R^+}\}.
\end{align*}
There are enough $(\sum, \wedge, \bigvee, \-{\#R^+})$-homomorphisms $\@K(E) -> \-{\#R^+}$ to separate points: we have an $(\sum, \wedge, \bigvee, \-{\#R^+})$-embedding
\begin{align*}
\eta : \@K(E) &--> \-{\#R^+}^{\EINV^\sigma_E} \\
\~A &|--> (\mu |-> \mu(A)).
\end{align*}
In particular, $\@K(E)$ obeys all Horn axioms in the operations $\sum, \wedge, \bigvee, \-{\#R^+}$ that hold in $\-{\#R^+}$.  \qed
\end{theorem}

\subsection{The algebra $\@L(E)$}

We next consider an algebra $\@L(E)$ closely related to $\@K(E)$.  As before, here $(X, E)$ is a compressible countable Borel equivalence relation.

Let $\@C(X)$ denote the set of Borel maps $X -> \-{\#R^+} = [0, \infty]$; we think of $\alpha \in \@C(X)$ as a ``weighted Borel subset'' of $X$.  Given $\alpha, \beta \in \@C(X)$, an \defn{$E$-equidecomposition}
\begin{align*}
\phi : \alpha \sim_E \beta
\end{align*}
is a Borel map $\phi : E -> \-{\#R^+}$ (where $E \subseteq X^2$) such that
\begin{align*}
\alpha(x) = \sum_{y \mathrel{E} x} \phi(x, y) =: \dom(\phi)(x), &&
\beta(y) = \sum_{x \mathrel{E} y} \phi(x, y) =: \rng(\phi)(y);
\end{align*}
$\alpha, \beta$ are \defn{$E$-equidecomposable}, written $\alpha \sim_E \beta$, if there is some $\phi : \alpha \sim_E \beta$.
We may think of such $\phi$ as a two-dimensional Borel ``matrix'' on each $E$-class $C$, whose row and column sums yield $\alpha, \beta$ respectively.

Our goal in this section is to show that $\@L(E) := \@C(X)/{\sim_E}$ is a cardinal algebra satisfying analogous properties to those in \cref{thm:k}, and in fact is a ``completion'' of $\@K(E)$ by adjoining divisors for indivisible elements; see \cref{thm:l}(ii).  This will require several preliminary steps: note that it is not even obvious that $\sim_E$ is an equivalence relation.

The following technical lemma says that the doubly infinitary version of axiom \S\ref{sec:ca}(C) (which holds in any cardinal algebra \cite[2.1]{Tar}) holds ``in a Borel way'' in the cardinal algebra $\-{\#R^+}$, depicted as follows:
\begin{equation*}
\begin{array}{c|cccc}
& v(0) & v(1) & v(2) & \dotsb \\
\hline
u(0) & d(u, v)(0, 0) & d(u, v)(0, 1) & d(u, v)(0, 2) & \dotsb \\
u(1) & d(u, v)(1, 0) & d(u, v)(1, 1) & d(u, v)(1, 2) & \dotsb \\
u(2) & d(u, v)(2, 0) & d(u, v)(2, 1) & d(u, v)(2, 2) & \dotsb \\
\vdots & \vdots & \vdots & \vdots & \ddots
\end{array}
\end{equation*}

\begin{lemma}
\label{lm:r-refine-borel}
There is a Borel map
\begin{align*}
d : \{(u, v) \in \-{\#R^+}^\#N \times \-{\#R^+}^\#N \mid \sum_i u(i) = \sum_j v(j)\} --> \-{\#R^+}^{\#N^2}
\end{align*}
such that for all $u, v$,
\begin{align*}
u(i) = \sum_j d(u, v)(i, j), &&
v(j) = \sum_i d(u, v)(i, j).
\end{align*}
\end{lemma}
\begin{proof}
We define $d(u, v)$ by cases:
\begin{enumerate}

\item[(I)]  Suppose $u(i) = v(j) = \infty$ for some $i, j$.  Let $i_0, j_0$ be the least such.  Put
\begin{align*}
d(u, v)(i, j) := \begin{cases}
v(j) &\text{if $i = i_0$}, \\
u(i) &\text{if $j = j_0$}, \\
0 &\text{otherwise}.
\end{cases}
\end{align*}

\item[(II)]  Suppose $u(i) \in \{0, \infty\}$ for all $i$, while $v(j) < \infty$ for all $j$.  If $u = v = 0$ then put $d(u, v) := 0$.  Otherwise, from $\sum_i u(i) = \sum_j v(j)$ we have $\sum_j v(j) = \infty$.  Let $f : \#N -> \{i \mid u(i) = \infty\}$ take each value in the codomain infinitely often (clearly such $f$ can be found in a Borel way from $u$).  Put $k_0 := 0$, and inductively let $k_{l+1} > k_l$ be least such that $\sum_{j=k_l}^{k_{l+1}-1} v(j) > 1$ (using that $\sum_j v(j) = \infty$).  Put
\begin{align*}
d(u, v)(i, j) := \begin{cases}
v(j) &\text{if $u(i) = \infty$ and $\exists l\, (f(l) = i \AND k_l \le j < k_{l+1})$}, \\
0 &\text{otherwise}.
\end{cases}
\end{align*}
The picture (when $u(i) = \infty$ for $i = 0, 2$) is as follows:
\begin{equation*}
\begin{array}{c|cccccccc}
& v(0) & v(1) & v(2) & v(3) & v(4) & v(5) & v(6) & \dotsb \\
\hline
\infty & v(0) & v(1) & & & & v(5) & & \dotsb \\
0 & \\
\infty & & & v(2) & v(3) & v(4) & & v(6) & \dotsb \\
\vdots &
\end{array}
\end{equation*}
In each row $i$ with $u(i) = \infty$, we put enough consecutive values of $v$ to achieve a sum $>1$ before switching to a different $i$; since each such row is visited infinitely often, its sum will be $\infty$.

The case where $u(i) < \infty$ for all $i$ and $v(j) \in \{0, \infty\}$ for all $j$ is symmetric.

\item[(III)]  Suppose $u(i), v(j) < \infty$ for all $i, j$.  Define $d(u, v)$ as follows:
\begin{itemize}
\item  Set $i_0 := j_0 := 0$.
\item  Inductively for each $k$, put
\begin{align*}
r_k := u(i_k) - \sum_{\substack{l < k \\ i_l = i_k}} d(u, v)(i_l, j_l), &&
s_k := v(j_k) - \sum_{\substack{l < k \\ j_l = j_k}} d(u, v)(i_l, j_l).
\end{align*}
If $r_k \le s_k$, put $d(u, v)(i_k, j_k) := r_k$ and $(i_{k+1}, j_{k+1}) := (i_k+1, j_k)$.  Otherwise, put $d(u, v)(i_k, j_k) := s_k$ and $(i_{k+1}, j_{k+1}) := (i_k, j_k+1)$.
\item  For all other $(i, j)$ not equal to some $(i_k, j_k)$, put $d(u, v)(i, j) := 0$.
\end{itemize}
The picture is as follows:
\begin{equation*}
\begin{array}{c|cccccccc}
& v(0) & v(1) & v(2) & v(3) & v(4) & v(5) & v(6) & \dotsb \\
\hline
u(0) & r_0 & & & & & & & \\
u(1) & r_1 & & & & & & & \\
u(2) & s_2 & s_3 & s_4 & r_5 & & & & \\
u(3) &     &     &     & s_6 & s_7 & r_8 & & \\
u(4) &     &     &     &     &     & r_9 & & \\
u(5) &     &     &     &     &     & s_{10} & \smash{\ddots} \\
\vdots &
\end{array}
\end{equation*}
The marked entries $(i, j)$ are $(i_0, j_0), (i_1, j_1), (i_2, j_2), \dotsc$.
At each step $(i_k, j_k)$, we place the largest possible value to ``fill out'' the remaining $u(i_k)$ or $v(j_k)$ (after subtracting the previous values in that row/column), and then move down or right in order to continue ``filling out'' the other value.
Using that $\sum_i u(i) = \sum_j v(j)$, it is straightforward to check that this works.
(Note that it is possible to have $\lim_{k -> \infty} i_k < \infty$ or $\lim_{k -> \infty} j_k < \infty$, if $\sum_i u(i) = \sum_j v(j) < \infty$.)

\item[(IV)]  In the remaining case, $u$ (say) takes infinite and nonzero finite values, while $v$ only takes finite values.  Let $u'(i) := u(i)$ if $u(i) < \infty$ and $u'(i) := 0$ otherwise, and put $u'' := u - u'$.  Define $v', v''$ such that $v = v' + v''$, $\sum_i u'(i) = \sum_j v'(j)$, and $\sum_i u''(i) = \sum_j v''(j)$, as follows: if $\sum_i u'(i) < \infty$, then let $k$ be least such that $\sum_{j \le k} v(j) > \sum_i u'(i)$, put
\begin{align*}
v'(j) := \begin{cases}
v(j) &\text{if $j < k$}, \\
\sum_i u'(i) - \sum_{j < k} v(j) &\text{if $j = k$}, \\
0 &\text{if $j > k$},
\end{cases}
\end{align*}
and $v'' := v - v'$; otherwise find $v', v''$ with $v = v' + v''$ and $\sum_j v'(j) = \sum_j v''(j) = \infty$ using a procedure similar to case (II).  We may then put $d(u, v) := d(u', v') + d(u'', v'')$, where the latter are computed using cases (II) and (III) above.
\qedhere

\end{enumerate}
\end{proof}

\begin{proposition}
\label{thm:l-trans}
$\sim_E$ is an equivalence relation on $\@C(X)$.
\end{proposition}
\begin{proof}
Reflexivity is easy: for $\alpha \in \@C(X)$, we have $\phi : \alpha \sim_E \alpha$ where $\phi(x, x) := \alpha(x)$ and $\phi(x, y) := 0$ for $x \ne y$.  Symmetry is obvious.  For transitivity, let $\alpha, \beta, \gamma \in \@C(X)$ with $\phi : \alpha \sim_E \beta$ and $\psi : \beta \sim_E \gamma$.  Let $(e^x_i)_{i \in \#N}$ for each $x \in X$ be an injective enumeration of $[x]_E$, Borel in $x$.  For $x \mathrel{E} y$, put
\begin{align*}
i^x_y := \text{the unique $i$ such that $e^x_i = y$};
\end{align*}
then clearly $(x, y) |-> i^x_y$ is a Borel map $E -> \#N$.  Define $\theta : E -> \-{\#R^+}$ by
\begin{align*}
\theta(x, z) := \sum_{y \mathrel{E} x} d((\phi(e^y_i, y))_i, (\psi(y, e^y_j))_j)(i^y_x, i^y_z)
\end{align*}
where $d$ is given by \cref{lm:r-refine-borel}.  Then
\begin{align*}
\dom(\theta)(x)
&= \sum_{z \mathrel{E} x} \sum_{y \mathrel{E} x} d((\phi(e^y_i, y))_i, (\psi(y, e^y_j))_j)(i^y_x, i^y_z) \\
&= \sum_{y \mathrel{E} x} \sum_{z \mathrel{E} y} d((\phi(e^y_i, y))_i, (\psi(y, e^y_j))_j)(i^y_x, i^y_z) \\
&= \sum_{y \mathrel{E} x} \sum_k d((\phi(e^y_i, y))_i, (\psi(y, e^y_j))_j)(i^y_x, k) \\
&= \sum_{y \mathrel{E} x} \phi(e^y_{i^y_x}, y) \\
&= \sum_{y \mathrel{E} x} \phi(x, y) \\
&= \alpha(x),
\end{align*}
and similarly $\rng(\theta) = \gamma$.  So $\theta : \alpha \sim_E \gamma$.
\end{proof}

We define
\begin{align*}
\@L(E) := \@C(X)/{\sim_E}.
\end{align*}
We equip $\@C(X)$ with the pointwise countable addition operation, with respect to which $\sim_E$ is a congruence relation (since if $\phi_i : \alpha_i \sim_E \beta_i$ for each $i$ then $\sum_i \phi_i : \sum_i \alpha_i \sim_E \sum_i \beta_i$).  Thus, countable addition on $\@C(X)$ descends to the quotient algebra $\@L(E)$.

\begin{lemma}
\label{thm:c-ca}
$\@C(X)$ is a cardinal algebra.
\end{lemma}
\begin{proof}
Axioms (A) and (B) are obvious.  For (C), given $\alpha, \beta, \gamma_i \in \@C(X)$ with $\alpha + \beta = \sum_i \gamma_i$, let
\begin{align*}
\alpha_i(x) &:= d((\alpha(x), \beta(x), 0, 0, \dotsc), (\gamma_i(x))_i)(0, i), \\
\beta_i(x) &:= d((\alpha(x), \beta(x), 0, 0, \dotsc), (\gamma_i(x))_i)(1, i),
\end{align*}
where $d$ is given by \cref{lm:r-refine-borel}; then $\alpha = \sum_i \alpha_i$, $\beta = \sum_i \beta_i$, and $\alpha_i + \beta_i = \gamma_i$ by the defining properties of $d$.  For (D), given $\alpha_i, \beta_i \in \@C(X)$ with $\alpha_i = \beta_i + \alpha_{i+1}$, put
\begin{align*}
\gamma(x) := \bigwedge_i \alpha_i(x);
\end{align*}
it is easily verified that $\alpha_i = \gamma + \sum_j \beta_{i+j}$.
\end{proof}

\begin{proposition}
\label{thm:l-ca}
$\@L(E)$ is a cardinal algebra.
\end{proposition}
\begin{proof}
By \cref{thm:c-ca} and \cite[6.10]{Tar}, it suffices to check that $\sim_E$ is a \defn{finitely refining} equivalence relation: that for $\alpha_1 + \alpha_2 = \alpha \sim_E \beta \in \@C(X)$, there are $\beta_1, \beta_2 \in \@C(X)$ such that $\beta = \beta_1 + \beta_2$, $\alpha_1 \sim_E \beta_1$, and $\alpha_2 \sim_E \beta_2$.  Let $\phi : \alpha \sim_E \beta$.  Let $(e^x_i)_{i \in \#N}$ for each $x \in X$ be an injective enumeration of $[x]_E$, Borel in $x$.  Define $\phi_1, \phi_2 : E -> \-{\#R^+}$ by
\begin{align*}
\phi_1(x, e^x_i) &:= d((\alpha_1(x), \alpha_2(x), 0, 0, \dotsc), (\phi(x, e^x_j))_j)(0, i), \\
\phi_2(x, e^x_i) &:= d((\alpha_1(x), \alpha_2(x), 0, 0, \dotsc), (\phi(x, e^x_j))_j)(1, i),
\end{align*}
where $d$ is given by \cref{lm:r-refine-borel}.  Then the definition of $d$ ensures that $\dom(\phi_1) = \alpha_1$, $\dom(\phi_2) = \alpha_2$, and $\phi_1 + \phi_2 = \phi$.  Put $\beta_1 := \rng(\phi_1)$ and $\beta_2 = \rng(\phi_2)$.
\end{proof}

For $\alpha \in \@C(X)$ and a (not necessarily $\sigma$-finite) $E$-invariant measure $\mu \in \INV^*_E$, put
\begin{align*}
\mu(\alpha) := \int \alpha \,d\mu.
\end{align*}

\begin{lemma}
\label{lm:eqdec-integ}
For $\alpha \sim_E \beta$, we have $\mu(\alpha) = \mu(\beta)$.
\end{lemma}
\begin{proof}
By invariance of $\mu$, we may define the measure $M$ on $E$ by
\begin{align*}
M(A) := \int \abs{A_x} \,d\mu(x) = \int \abs{A^y} \,d\mu(y)
\end{align*}
where $A_x := \{y \mid (x, y) \in A\}$ and $A^y := \{x \mid (x, y) \in A\}$; see e.g., \cite[\S16]{KM}.  Now letting $\phi : \alpha \sim_E \beta$, we have $\mu(\alpha) = \int \phi \,dM$.  Indeed, let $(e^x_i)_{i \in \#N}$ for each $x \in X$ be an injective enumeration of $[x]_E$, Borel in $x$, and let $E_i := \{(x, e^x_i) \mid x \in X\}$, so that $E = \bigsqcup_i E_i$; then
\begin{align*}
\mu(\alpha)
&= \int \sum_i \phi(x, e^x_i) \,d\mu(x) \\
&= \sum_i \int \phi(x, e^x_i) \,d\mu(x) \\
&= \sum_i \int_{(x, y) \in E_i} \phi(x, y) \,dM(x, y) \\
&= \int \phi(x, y) \,dM(x, y).
\end{align*}
Similarly, $\mu(\beta) = \int \phi \,dM$, whence $\mu(\alpha) = \mu(\beta)$.
\end{proof}

It follows that each $\mu \in \INV^*_E$ defines a map $\@L(E) -> \-{\#R^+}$, which is a $\sum$-homomorphism since integration is countably additive (by the monotone convergence theorem).  Thus, analogously to \cref{thm:k-homom-measure}, we have a map
\begin{align*}
\INV^*_E &`-> \{\text{$\sum$-homomorphisms } \@L(E) -> \-{\#R^+}\} \\
\mu &|--> (\~\alpha |-> \int \alpha \,d\mu)
\end{align*}
which is in fact a bijection (see \cref{thm:l}(iii) below).  We also have (analogously to \cref{thm:k-embed}) a $\sum$-homomorphism
\begin{align*}
\iota : \@L(E) &--> \-{\#R^+}^{\EINV^\sigma_E} \\
\~\alpha &|--> (\mu |-> \mu(\alpha));
\end{align*}
we will show below that it preserves finite meets (hence countable joins) and is an embedding.

We now begin the comparison between $\@K(E)$ and $\@L(E)$.  Given a Borel set $A \in \@B(X)$, its characteristic function $\chi_A$ belongs to $\@C(X)$; and if $A, B \in \@B(X)$ and $f : A \sim_E B$ is an equidecomposition, then the characteristic function of the graph of $f$ is an equidecomposition $\chi_A \sim_E \chi_B$.  Thus $A |-> \chi_A$ descends to a map between the quotients
\begin{align*}
\chi : \@K(E) &--> \@L(E) \\
\~A &|--> \~\chi_A
\end{align*}
which clearly preserves countable sums, i.e., is a $\sum$-homomorphism.

\begin{proposition}
\label{thm:chi-embed}
$\chi$ is an order-embedding.
\end{proposition}
(Note that this is not obvious: an equidecomposition $\chi_A \sim_E \chi_B$ need not be the characteristic function of the graph of an equidecomposition $A \sim_E B$.)
\begin{proof}
We have a commutative diagram
\begin{equation*}
\begin{tikzcd}
\@K(E) \drar["\eta"'] \rar["\chi"] & \@L(E) \dar["\iota"] \\
& \-{\#R^+}^{\EINV^\sigma_E}
\end{tikzcd}
\end{equation*}
where $\eta$ is from \cref{thm:k-embed} and $\iota$ is from above.  Since $\eta$ is an order-embedding and $\iota$ is order-preserving, it follows that $\chi$ is an order-embedding.
\end{proof}

For any $\alpha \in \@C(X)$, put
\begin{align*}
\sum_E \alpha : X/E &--> \-{\#R^+} \\
C &|--> \sum_{x \in C} \alpha(x).
\end{align*}
We say that $\~\alpha \in \@L(E)$ (or $\alpha \in \@C(X)$) is \defn{($E$-)finite} if $\alpha$ has finite sum on every $E$-class (i.e., $\sum_E \alpha : X/E -> [0, \infty)$), and \defn{($E$-)aperiodic} if $\alpha$ has sum $0$ or $\infty$ on every $E$-class (i.e., $\sum_E \alpha : X/E -> \{0, \infty\}$).  We let $\@L^\fin(E), \@L^\ap(E) \subseteq \@L(E)$ denote the subsets of finite, respectively aperiodic, elements.  Clearly $\chi(\@K^\fin(E)) \subseteq \@L^\fin(E)$ and $\chi(\@K^\ap(E)) \subseteq \@L^\ap(E)$.

\begin{lemma}
\label{lm:lfin}
Suppose $\alpha \in \@C(X)$ is $E$-finite.  Then $\alpha$ has $E$-smooth support, i.e., $E|\alpha^{-1}((0, \infty])$ is smooth.  Moreover, for any $\beta \in \@C(X)$ with $\sum_E \alpha \le \sum_E \beta$, we have $\~\alpha \le \~\beta$.
\end{lemma}
\begin{proof}
If $\alpha(x) > 0$, then since $\sum_{y \mathrel{E} x} \alpha(y) < \infty$, the set $\{y \mathrel{E} x \mid \alpha(y) > 1/n\}$ is finite for each $n$ and nonempty for some $n$; this easily implies that $\alpha$ has smooth support.

For any $\alpha \in \@C(X)$ with smooth support (not necessarily $E$-finite), letting $A \subseteq \alpha^{-1}((0, \infty])$ be a Borel transversal of $E|\alpha^{-1}((0, \infty])$, it is easily seen that $\alpha \sim_E \alpha'$, where $\alpha'(x) := \sum_{y \mathrel{E} x} \alpha(y)$ for $x \in A$ and $\alpha'(x) := 0$ for $x \not\in A$, so that $\alpha'$ is nonzero on at most one point per $E$-class.  Now if $\sum_E \alpha \le \sum_E \beta$, then $\gamma := \beta|[\alpha^{-1}((0, \infty])]_E$ also has smooth support, and $\sum_E \alpha \le \sum_E \gamma$; letting $\gamma' \sim_E \gamma$ be nonzero on at most one point per $E$-class, we have $\sum_E \alpha' = \sum_E \alpha \le \sum_E \gamma = \sum_E \gamma'$, which easily implies $\~\alpha' \le \~\gamma'$, whence $\~\alpha = \~\alpha' \le \~\gamma' = \~\gamma \le \~\beta$.
\end{proof}

\begin{proposition}
\label{thm:lfin}
We have an order-isomorphism
\begin{align*}
\sum_E : \@L^\fin(E) &\cong \{\text{Borel maps $X/E -> [0, \infty)$ with smooth support}\}
\end{align*}
(where $f : X/E -> [0, \infty)$ having smooth support means that $E|\bigcup f^{-1}((0, \infty))$ is smooth).
\end{proposition}
\begin{proof}
By \cref{lm:lfin}, $\sum_E$ is an order-embedding.
For surjectivity, given Borel $f : X/E -> [0, \infty)$ with smooth support, letting $A$ be a Borel transversal of $E|\bigcup f^{-1}((0, \infty))$, we have $f = \sum_E \alpha$ where $\alpha(x) := f([x]_E)$ for $x \in A$ and $\alpha(x) := 0$ for $x \not\in A$.
\end{proof}

\begin{lemma}
\label{thm:l-csec}
For every $\alpha \in \@C(X)$ and $E$-complete section $Y \subseteq X$, there is a $\beta \in \@C(X)$ supported on $Y$ such that $\alpha \sim_E \beta$.
\end{lemma}
\begin{proof}
Let $f : X -> Y$ be Borel with $f(x) \mathrel{E} x$.  Put $\beta(y) := \sum_{x \in f^{-1}(y)} \alpha(x)$ for $y \in Y$ and $\beta(x) = 0$ for $x \not\in Y$.  Put $\phi(x, f(x)) := \alpha(x)$ and $\phi(x, y) := 0$ for $y \in [x]_E \setminus \{f(x)\}$.  Then $\phi : \alpha \sim_E \beta$.
\end{proof}

\begin{proposition}
\label{thm:chi-aperiodic}
$\chi : \@K^\ap(E) -> \@L^\ap(E)$ is an isomorphism.
\end{proposition}
\begin{proof}
By \cref{thm:chi-embed}, it remains to show surjectivity.  Let $\alpha \in \@C(X)$ with sum $0$ or $\infty$ on each $E$-class; we must find an $A \in \@B(X)$ such that $\alpha \sim_E \chi_A$.

First, we claim that we may assume that $\alpha$ only takes values in $\{0\} \cup \{2^{-n} \mid n \in \#N\}$.  By compressibility of $E$, we may assume that $(X, E) = (Y \times 2 \times \#N, F \times I_2 \times I_\#N)$ for some $(Y, F)$ (where $I_2, I_\#N$ are the indiscrete equivalence relations $2 \times 2, \#N \times \#N$).  By \cref{thm:l-csec}, we may assume that $\alpha$ is supported on $Y \times \{0\} \times \{0\}$.  Now for each $y \in Y$, ``spread out'' $\alpha(y, 0, 0)$ according to its binary expansion along $\{y\} \times 2 \times \#N$ to get $\beta \in \@C(X)$.  That is, if $\alpha(y, 0, 0) = \infty$ then put $\beta(y, i, j) := 1$ for all $i, j$; otherwise, let $\alpha(y, 0, 0) = a.b_1b_2b_3\dotsm$ be the binary expansion, put $\beta(y, 0, i) := 1$ for $i < a$ and $\beta(y, 0, i) := 0$ for $i \ge a$, and put $\beta(y, 1, 0) := 0$ and $\beta(y, 1, i) := b_i 2^{-i}$ for $i > 0$.  Then clearly $\alpha \sim_E \beta$ and $\beta$ only takes values in $\{0\} \cup \{2^{-n} \mid n \in \#N\}$, so we may replace $\alpha$ by $\beta$.

Now, the union $A \subseteq X$ of those $E$-classes $C$ such that $C \cap \alpha^{-1}(2^{-n})$ is nonempty finite for some $n$ is clearly smooth, and so we easily have $\chi_A \sim_E \alpha|A$ (e.g., because $\chi_A \sim_E \beta \sim_E \alpha|A$ where $\beta$ is $\infty$ on a single point in each $E|A$-class).  So we may assume that for each $n$, $E|\alpha^{-1}(2^{-n})$ is aperiodic.  For each $n$, using \cref{lm:aperiodic-fse}, let $F_n$ be a finite Borel subequivalence relation of $E|\alpha^{-1}(2^{-n})$ with all classes of size $2^n$, and let $A_n \subseteq \alpha^{-1}(2^{-n})$ be a Borel transversal of $F_n$.  Then it is easily seen that $\chi_{A_n} \sim_E \alpha|\alpha^{-1}(2^{-n})$, whence putting $A := \bigcup_n A_n$, we have $\chi_A \sim_E \alpha$.
\end{proof}

Using \cref{thm:lfin,thm:chi-aperiodic}, we now transfer most of the properties of $\@K(E)$ to $\@L(E)$, yielding the analogue of \cref{thm:k} for $\@L(E)$:

\begin{theorem}
\label{thm:l}
\begin{enumerate}
\item[(i)]  $\@L(E)$ is a cardinal algebra with finite meets, countable joins, and real multiples of all elements.
\item[(ii)]  The embedding $\chi : \@K(E) `-> \@L(E)$ preserves finite meets and countable joins, and restricts to an isomorphism $\@K^\ap(E) \cong \@L^\ap(E)$.  Furthermore, the closure of the image of $\chi$ under real multiples and countable sums is all of $\@L(E)$.
\item[(iii)]  We have canonical bijections (where $\bigvee^\up$ denotes countable increasing joins)
\begin{align*}
\INV^*_E &\cong \{\text{$\sum$-homomorphisms } \@L(E) -> \-{\#R^+}\} \\
&= \{\text{$(\sum, \bigvee^\up, \-{\#R^+})$-homomorphisms } \@L(E) -> \-{\#R^+}\}, \\
\EINV^*_E &\cong \{\text{$(\sum, \wedge, \bigvee, \-{\#R^+})$-homomorphisms } \@L(E) -> \-{\#R^+}\}
\end{align*}
compatible with those for $\@K(E)$ from \cref{thm:k}.
\item[(iv)]  We have an $(\sum, \wedge, \bigvee, \-{\#R^+})$-embedding
\begin{align*}
\iota : \@L(E) &--> \-{\#R^+}^{\EINV^\sigma_E} \\
\~\alpha &|--> (\mu |-> \mu(\alpha))
\end{align*}
extending $\eta : \@K(E) -> \-{\#R^+}^{\EINV^\sigma_E}$.  In particular, $\@L(E)$ obeys all Horn axioms in the operations $(\sum, \wedge, \bigvee, \-{\#R^+})$ that hold in $\-{\#R^+}$.
\end{enumerate}
\end{theorem}
\begin{proof}
(i): $\@L(E)$ is a cardinal algebra by \cref{thm:l-ca}, and clearly has real multiples inherited from $\@C(X)$ (by \cref{lm:ca-homom-sfiltjoin-rmult}); it remains to construct finite meets.  The greatest element of $\@L(E)$ is $\infty := \~\infty$ where $\infty \in \@C(X)$ is the constantly $\infty$ function.  To compute the meet of $\~\alpha, \~\beta$: let $A, B \subseteq X$ be the unions of the $E$-classes on which $\alpha, \beta$ respectively have finite sum, so that $\alpha|A, \beta|B$ are finite while $\alpha|(X \setminus A), \beta|(X \setminus B)$ are aperiodic.  Using \cref{thm:lfin}, the meet of $[\alpha|(A \cap B)]_{\sim_E}, [\beta|(A \cap B)]_{\sim_E}$ is given by $\~\gamma$ where $\sum_E \gamma = \sum_E \alpha|(A \cap B) \wedge \sum_E \beta|(A \cap B)$.  Using \cref{lm:lfin}, the meet of $[\alpha|(A \setminus B)]_{\sim_E}, [\beta|(A \setminus B)]_{\sim_E}$ is the former (since $\beta$ has infinite sum on every $E|(A \setminus B)$-class), and similarly the meet of $[\alpha|(B \setminus A)]_{\sim_E}, [\beta|(B \setminus A)]_{\sim_E}$ is the latter.  Using \cref{thm:chi-aperiodic}, the meet of $[\alpha|(X \setminus (A \cup B))]_{\sim_E}, [\beta|(X \setminus (A \cup B))]_{\sim_E}$ may be computed in $\@K(E)$.  The sum of these four meets is $\~\alpha \wedge \~\beta$.

Note that binary joins in $\@L(E)$ may be computed in a similar manner.

(ii): It is easily verified that the above procedure for computing binary meets and joins in $\@L(E)$ agrees, when $\alpha = \chi_A$ and $\beta = \chi_B$, with the computation in $\@K(E)$.  Using that $E$ is compressible, the greatest element $\infty = \~\infty \in \@L(E)$ is equal to $\~1 = \chi(\~X) = \chi(\infty)$.  So $\chi$ preserves finite meets and (by \cref{lm:ca-homom-sfiltjoin-rmult}) countable joins.  That $\chi$ restricts to an isomorphism $\@K^\ap(E) \cong \@L^\ap(E)$ is \cref{thm:chi-aperiodic}.  For every $\alpha \in \@C(X)$, we can write $\alpha$ as a countable real linear combination $\sum_i r_i \chi_{A_i}$ of characteristic functions of $A_i \in \@B(X)$, whence $\~\alpha = \sum_i r_i \~\chi_{A_i}$; thus the image of $\chi$ generates $\@L(E)$ under real multiples and countable sums.

(iii): By \cref{lm:eqdec-integ} and the succeeding remarks, we have a commutative diagram
\begin{equation*}
\begin{tikzcd}
\INV^*_E \rar \drar["\cong"'] & \{\text{$\sum$-homomorphisms } \@L(E) -> \-{\#R^+}\} \dar["(-)\vert\@K(E)"] \\
& \{\text{$\sum$-homomorphisms } \@K(E) -> \-{\#R^+}\}
\end{tikzcd}
\end{equation*}
The vertical map is injective, since $\@K(E)$ generates $\@L(E)$ under countable sums and real multiples (by (ii)) and $\sum$-homomorphisms $\@L(E) -> \-{\#R^+}$ preserve real multiples (by \cref{lm:ca-homom-sfiltjoin-rmult}).  It follows that the horizontal map is bijective, yielding the first bijection in (iii).  For the second, a $(\sum, \wedge, \bigvee, \-{\#R^+})$-homomorphism $\@L(E) -> \-{\#R^+}$ still preserves finite meets when restricted to $\@K(E)$ by (ii), hence corresponds to an ergodic measure by \cref{lm:k-meas-ergodic}; and conversely, it is easily seen from the computation of binary meets in (i) that an ergodic measure $\mu : \@L(E) -> \-{\#R^+}$ preserves binary meets (hence countable joins and real multiples, by \cref{lm:ca-homom-sfiltjoin-rmult,lm:ca-homom-meetjoin}).

(iv): It suffices to show that for $\~\alpha \not\le \~\beta \in \@L(E)$, there is $\mu \in \EINV^\sigma_E$ such that $\mu(\alpha) > \mu(\beta)$.  By restricting $E$, we may assume that each of $\~\alpha, \~\beta$ is either finite or aperiodic.  If both are aperiodic, apply \cref{thm:chi-aperiodic} and \cref{lm:k-embed}.  If $\~\alpha$ is finite, since $\~\alpha \not\le \~\beta$, by \cref{lm:lfin} there is an $E$-class $C$ such that $\sum_{x \in C} \alpha(x) > \sum_{x \in C} \beta(x)$; let $\mu$ be an atomic measure on $C$.  If $\~\alpha$ is aperiodic while $\~\beta$ is finite, since $\~\alpha \not\le \~\beta$, there is an $E$-class $C$ such that $\sum_{x \in C} \alpha(x) = \infty$; let $\mu$ be an atomic measure on $C$.
\end{proof}

\subsection{The duality theorem}

Regarding $\@K(E)$ and $\@L(E)$ as algebras under the operations $\sum, \wedge, \bigvee, \-{\#R^+}$, \cref{thm:k-embed} and \cref{thm:l}(iv) say that these algebras admit enough homomorphisms to $\-{\#R^+}$ to separate points.  It is thus natural to regard the space of all such homomorphisms, i.e., $\EINV^*_E$, as the ``dual'' or ``spectrum'' of the algebra, and to ask whether we may recover the algebra as an ``algebra of functions'' on the space $\EINV^*_E$ equipped with suitable structure.

We give in this section a positive answer, subject to some technical caveats.  First, since every element of $\-{\#R^+}$ is completely divisible, so will be every element of an ``algebra of functions'' with values in $\-{\#R^+}$; thus we can only hope to recover $\@L(E)$, not $\@K(E)$.  Second, as mentioned previously, non-$\sigma$-finite measures in $\EINV^*_E$ are not so tractable; thus we will consider instead the subspace $\EINV^\sigma_E \subseteq \EINV^*_E$ as the ``dual'' of $\@L(E)$.  Finally, there is the question of what kind of ``space'' $\EINV^\sigma_E$ is.  It is not enough to regard it as a (nonstandard) Borel space, due to \cref{rmk:einvs-borel} below.

A \defn{$\sigma$-topology} on a set $X$ is a collection of subsets of $X$ (called \defn{$\sigma$-open}), closed under countable unions and finite intersections; a \defn{$\sigma$-topological space} is a set equipped with a $\sigma$-topology.  A \defn{$\sigma$-continuous map} between $\sigma$-topological spaces is a map such that the preimage of every $\sigma$-open set is $\sigma$-open.  The notions of \defn{product $\sigma$-topology} and \defn{subspace $\sigma$-topology} are defined in the usual manner (i.e., the smallest $\sigma$-topology making the projection maps, respectively the inclusion, $\sigma$-continuous).  Note that a $\sigma$-topology generates both a topology (by closing under arbitrary unions) and a $\sigma$-algebra (by closing under complements and countable unions), hence contains more information than both a topology and a Borel structure.  In particular, every $\sigma$-continuous map is Borel with respect to the induced Borel structures.

We equip $\-{\#R^+} = [0, \infty]$ with the $\sigma$-topology whose nontrivial $\sigma$-open sets are $(r, \infty]$ for $r \in (0, \infty)$.  We view $\EINV^\sigma_E$ as a $\sigma$-topological subspace of the product space $\-{\#R^+}^{\@B(X)}$; thus, the $\sigma$-topology on $\EINV^\sigma_E$ is generated by the subbasic $\sigma$-open sets
\begin{align*}
U_{A,r} := \{\mu \in \EINV^\sigma_E \mid \mu(A) > r\}
\end{align*}
for $A \in \@B(X)$ and $r \in (0, \infty)$.  In addition, we also equip $\EINV^\sigma_E$ with the multiplication action of the multiplicative monoid $(0, \infty)$. 

\begin{lemma}
For every $\~\alpha \in \@L(E)$, the map
\begin{align*}
\iota(\~\alpha) : \EINV^\sigma_E &--> \-{\#R^+} \\
\mu &|--> \mu(\alpha)
\end{align*}
is $\sigma$-continuous and $(0, \infty)$-equivariant.
\end{lemma}
\begin{proof}
$(0, \infty)$-equivariance is obvious.  For $\sigma$-continuity, write $\alpha = \sum_i r_i \chi_{A_i}$ as a countable linear combination of positive real multiples of characteristic functions of $A_i \in \@B(X)$, so that $\mu(\alpha) = \sum_i r_i \mu(A_i)$; $\sigma$-continuity of $\iota(\~\alpha)$ thus follows from $\sigma$-continuity of $r_i \cdot (-) : \-{\#R^+} -> \-{\#R^+}$, which is obvious, and of $\sum : \-{\#R^+}^\#N -> \-{\#R^+}$, which is straightforward (since $\sum_i r_i > s \iff \exists n\; \exists i_1, \dotsc, i_n \in \#N\; \exists q_1, \dotsc, q_n \in \#Q\, (q_1 + \dotsb + q_n > s \AND r_{i_1} > q_1 \AND \dotsb \AND r_{i_n} > q_n)$).
\end{proof}

In other words, the map $\iota : \@L(E) -> \-{\#R^+}^{\EINV^\sigma_E}$ from \cref{thm:l}(iv) lands in the subalgebra of $\sigma$-continuous, $(0, \infty)$-equivariant maps $\EINV^\sigma_E -> \-{\#R^+}$.  We now have the following duality theorem:

\begin{theorem}
\label{thm:l-duality}
The map
\begin{align*}
\iota : \@L(E) &--> \{\text{$\sigma$-continuous, $(0, \infty)$-equivariant maps } \EINV^\sigma_E -> \-{\#R^+}\}
\end{align*}
is an $(\sum, \wedge, \bigvee, \-{\#R^+})$-isomorphism.
\end{theorem}

\begin{remark}
\label{rmk:einvs-borel}
\Cref{thm:l-duality} fails if we replace ``$\sigma$-continuous'' with ``Borel'' (with the nonstandard Borel structure on $\EINV^\sigma_E$ generated by the maps $\mu |-> \mu(A)$ for $A \in \@B(X)$).  Indeed, let $(X, E) = (\#R, E_v)$ where $E_v$ is the Vitali equivalence relation, and let $f : \EINV^\sigma_E -> \-{\#R^+}$ be given by $f(\mu) := 0$ if $\mu((0, 1)) < \infty$, else $f(\mu) := \infty$.  Clearly $f$ is Borel and $(0, \infty)$-equivariant.  Suppose we had $f = \iota(\~\alpha)$ for some $\~\alpha \in \@L(E)$.  For atomic $\mu \in \EINV^\sigma_E$, we have $\mu((0, 1)) = \infty$, whence $\infty = f(\mu) = \iota(\~\alpha)(\mu) = \int \alpha \,d\mu$; so $\alpha$ must be nonzero on every coset of $\#Q$.  But then for Lebesgue measure $\mu$, we have $\iota(\~\alpha)(\mu) = \int \alpha \,d\mu > 0 = f(\mu)$, a contradiction.
\end{remark}

Recall from \cref{thm:meas-extend} that for a Borel set $A \in \@B(X)$ and an $E|A$-ergodic invariant probability measure $\mu \in \EINV_{E|A}$,
\begin{align*}
[\mu]_E \in \EINV^\sigma_E
\end{align*}
denotes the unique $E$-ergodic invariant extension of $\mu$.  From the definition of $[\mu]_E$ in \cref{thm:meas-extend}, we clearly have

\begin{lemma}
\label{lm:measure-ext-borel}
The map $[-]_E : \EINV_{E|A} -> \EINV^\sigma_E$ is Borel (where $\EINV_{E|A}$ has the usual standard Borel structure).  \qed
\end{lemma}

\begin{proof}[Proof of \cref{thm:l-duality}]
By \cref{thm:l}(iv), it remains only to show that $\iota$ is surjective.
Let $f : \EINV^\sigma_E -> \-{\#R^+}$ be $\sigma$-continuous and $(0, \infty)$-equivariant.  So for each rational $r \in (0, \infty)$, $f^{-1}((r, \infty]) \subseteq \EINV^\sigma_E$ is a countable union of finite intersections of subbasic $\sigma$-open sets $U_{A,r}$; let $\@A$ denote all countably many $A \in \@B(X)$ involved in these expressions (for all $r$).  In particular, $f(\mu)$ can only depend on the values of $\mu(A)$ for $A \in \@A$.

Let $\Gamma$ be a countable group with a Borel action on $X$ inducing $E$.  Equip $X$ with a Polish topology given by \cref{thm:topmeasdecomp}, making every $A \in \@A$ open and making the $\Gamma$-action continuous, so that we have a topological ergodic decomposition $p : X ->> X//E$ which is also a measure-theoretic ergodic decomposition.  Let the Borel sets $R \subseteq X//E$ and $S \subseteq X$ and the Borel map $(C |-> \mu_C) : R -> \EINV_{E|S}$ be given by \cref{thm:topmeasdecomp}.

For $C \in X//E \setminus R$, let $\nu_C \in \EINV^\sigma_E$ be an $E$-ergodic invariant $\sigma$-finite measure supported on $C$, such that $C |-> \nu_C$ is Borel; for example, compose a Borel section of $p : X ->> X//E$ (which exists by \cref{thm:topdecomp}) with a Borel map $X -> \EINV^\sigma_E$ taking a point $x$ to an atomic measure on $[x]_E$.  Note that by definition of $R$, each $\nu_C$ is necessarily totally singular.

Note that for any totally singular $\mu \in \EINV^\sigma_E$, we have $\mu(A) = 2\mu(A)$ for all $A \in \@A$, whence $f(\mu) = f(2\mu) = 2f(\mu)$ since $f$ is $(0, \infty)$-equivariant, whence $f(\mu) \in \{0, \infty\}$.

Now define $\alpha \in \@C(X)$ by
\begin{align*}
\alpha(x) := \begin{cases}
f([\mu_{\den{x}_E}]_E) &\text{if $x \in S$}, \\
0 &\text{if $x \in p^{-1}(R) \setminus S$}, \\
f(\nu_{\den{x}_E}) &\text{if $x \not\in p^{-1}(R)$}.
\end{cases}
\end{align*}
By \cref{lm:measure-ext-borel}, this is Borel.  We claim that $f = \iota(\~\alpha)$, i.e., $f(\mu) = \iota(\~\alpha)(\mu) = \int \alpha \,d\mu$ for all $\mu \in \EINV^\sigma_E$.  Note that by ergodicity, each $\mu$ is supported on some $C \in X//E$.
\begin{itemize}

\item  Suppose $\mu$ is non-totally-singular, hence supported on some $C \in R$.  Then (by \cref{thm:topmeasdecomp}(ii)) $\mu = r [\mu_C]_E$ for some $r \in (0, \infty)$.  We have
\begin{align*}
\begin{aligned}
\int \alpha \,d\mu
&= r \int \alpha \,d[\mu_C]_E \\
&= r f([\mu_C]_E) [\mu_C]_E(S \cap C) &&\text{by definition of $\alpha$} \\
&= f(r [\mu_C]_E) &&\text{by $(0, \infty)$-equivariance of $f$} \\
&= f(\mu),
\end{aligned}
\end{align*}
as desired.

\item  Suppose $\mu$ is totally singular and supported on some $C \in X//E \setminus R$.  Since $\mu$ is totally singular, $f(\mu) \in \{0, \infty\}$.  For each $A \in \@A$, since $A \subseteq X$ is open (and so meets $C$ iff it meets every $E$-class in $C$), we have $\mu(A) = \nu_C(A) = \infty$ if $A \cap C \ne \emptyset$, otherwise $\mu(A) = \nu_C(A) = 0$; thus $f(\mu) = f(\nu_C) \in \{0, \infty\}$.  So by definition of $\alpha$, $\int \alpha \,d\mu = f(\nu_C) \cdot \mu(C) = f(\nu_C) \cdot \infty = f(\mu)$.

\item  Finally, suppose $\mu$ is totally singular and supported on some $C \in R$.  We again have $f(\mu) \in \{0, \infty\}$, while by definition of $\alpha$,
\begin{align*}
\int \alpha \,d\mu
= f([\mu_C])_E \cdot \mu(S)
= \begin{cases}
0 &\text{if $f([\mu_C]_E) = 0$}, \\
\infty &\text{if $f([\mu_C]_E) > 0$}.
\end{cases}
\end{align*}
So we must show that $f(\mu) > 0$ iff $f([\mu_C]_E) > 0$.  By definition of $\@A$, we may write
\begin{align*}
f^{-1}((0, \infty]) = \bigcup_{i \in I} \bigcap_{j \in J_i} U_{A_{ij}, r_{ij}}
\end{align*}
where $I$ is countable, each $J_i$ is finite, each $A_{ij} \in \@A$, and each $r_{ij} \in (0, \infty)$.  If $[\mu_C]_E \in U_{A_{ij}, r_{ij}}$ for some $i, j$, i.e., $[\mu_C]_E(A_{ij}) > r_{ij}$, then $A_{ij}$ must meet $C$, whence $\mu(C) = \infty > r_{ij}$, whence $\mu \in U_{A_{ij}, r_{ij}}$; thus $f([\mu_C]_E) > 0$ implies $f(\mu) > 0$.  Conversely, if $f(\mu) > 0$, then $\mu \in \bigcap_{j \in J_i} U_{A_{ij}, r_{ij}}$ for some $i$, whence each $A_{ij}$ for $j \in J_i$ must meet $C$, whence $[\mu_C]_E(A_{ij}) > 0$ for each $j \in J_i$; since $J_i$ is finite, there is some $r \in (0, \infty)$ such that $r[\mu_C]_E(A_{ij}) > r_{ij}$ for each $j \in J_i$, whence $r[\mu_C]_E \in \bigcap_{j \in J_i} U_{A_{ij}, r_{ij}}$, whence $rf([\mu_C]_E) = f(r[\mu_C]_E) > 0$, whence $f([\mu_C]_E) > 0$.  Thus $f(\mu) > 0$ iff $f([\mu_C]_E) > 0$, as desired.
\qedhere

\end{itemize}
\end{proof}

\subsection{Other algebraic operations}

We conclude by briefly considering the (non)existence of other canonical algebraic operations on $\@K(E)$ and $\@L(E)$.

\begin{remark}
Countable decreasing meets do not exist in general; indeed, a countable decreasing sequence in $\@K(E)$ need not have a meet in $\@L(E)$.  Consider the Vitali equivalence relation $(\#R, E_v)$.  For each $n \ge 1$, let $A_n := [0, 1/n) \in \@B(\#R)$.  Then for Lebesgue measure $\mu$, we have $\mu(A_n) = 1/n$; so if $\bigwedge_n \~\chi_{A_n}$ existed, we must have $\mu(\bigwedge_n \~\chi_{A_n}) = 0$.  But for each $x \in \#R$, we clearly have $\#Q + x \preceq_E A_n$ for all $n$, whence $\~\chi_{\#Q + x} \le \bigwedge_n \~\chi_{A_n}$.  So $\bigwedge_n \~\chi_{A_n}$ must be represented by a function in $\@C(X)$ which is nonzero on each $E_v$-class, whence $\mu(\bigwedge_n \~\chi_{A_n}) > 0$, a contradiction.
\end{remark}

\begin{remark}
Given $\~A \le \~B \in \@K(E)$, there is by definition some $\~C$ such that $\~A + \~C = \~B$.  However, there does not seem to be a canonical choice of such a $\~C$ that works for all $\~A, \~B$.  In other words, there does not seem to be a canonical way of defining a partial ``difference'' operation $\~B - \~A$ for all $\~A \le \~B$, such that $\~A + (\~B - \~A) = \~B$.  The same is true in $\@L(E)$.

In particular, there is not always a smallest or largest $\~C$.  Consider the tail equivalence relation $(2^\#N, E_t)$ and a Borel complete section $A \subseteq 2^\#N$ with $E_t|A \cong E_0$ (where $E_0$ is equality modulo finite on $2^\#N$; the existence of such a complete section $A$ is standard, following for example from \cite[9.3]{DJK}).  Then there is no smallest $\~\alpha \in \@L(E_t)$ (or $\~A' \in \@K(E_t)$) such that $\~\chi_A + \~\alpha = \infty$ (or $\~A + \~A' = \infty$): the union $B$ of the $E_t$-classes on which $\alpha$ is zero must be such that $E_t|(A \cap B)$ is compressible, or else we could not have $\~\chi_A + \~\alpha = \infty$; and given such $\alpha$, we can always make $\alpha$ zero on a single class outside $B$ to get a strictly smaller $\~\beta < \~\alpha$ with $\~\chi_A + \~\beta = \infty$.  And there is no largest $\~\alpha \in \@L(E_t)$ (or $\~A' \in \@K(E_t)$) such that $\~\chi_A + \~\alpha = \~\chi_A$ (or $\~A + \~A' = \~A$): such $\alpha$ are precisely those for which the union $C$ of the $E_t$-classes on which $\alpha$ is nonzero has $E_t|(A \cap C)$ compressible.
\end{remark}

\appendix

\section{Inverse limits of quasi-Polish spaces}
\label{app:lim}

We prove here some technical results regarding inverse limits of quasi-Polish spaces.

\begin{proposition}
\label{thm:dense-cofiltlim}
Let $X_0 <---{f_0} X_1 <---{f_1} X_2 <---{f_2} \dotsb$ be a sequence of continuous maps between quasi-Polish spaces, such that each $f_i$ has dense image.  Then each projection map $p_i : \projlim_i X_i -> X_i$ has dense image.
\end{proposition}
\begin{proof}
The \defn{lax colimit} of the sequence is the space
\begin{align*}
X' := \bigsqcup_i X_i \sqcup \projlim_i X_i,
\end{align*}
with topology given by the basic open sets
\begin{align*}
\Up{U} := \bigcup_{j \ge i} (f_i \circ \dotsb \circ f_{j-1})^{-1}(U) \cup p_i^{-1}(U) \qquad\text{for $i \in \#N$ and open $U \subseteq X_i$}.
\end{align*}

\begin{lemma}
$X'$ is quasi-Polish.
\end{lemma}
\begin{proof}
Given a quasi-Polish space $X$, the space
\begin{align*}
X_\bot := X \sqcup \{\bot\}
\end{align*}
with open sets consisting of open sets in $X$ together with all of $X_\bot$ is easily seen to be quasi-Polish; see e.g., \cite[3.4]{Ch}.  We claim that we have a homeomorphism
\begin{align*}
X' &--> \{(x_i)_i \in \prod_i (X_i)_\bot \mid x_0 \in X_0 \AND \forall i\, (x_{i+1} \in X_{i+1} \implies x_i = f_i(x_{i+1}))\} \\
x \in X_i &|--> ((f_0 \circ \dotsb \circ f_{i-1})(x), \dotsc, f_{i-1}(x), x, \bot, \bot, \dotsc), \\
y \in \projlim_i X_i &|--> (p_0(y), p_1(y), \dotsc).
\end{align*}
This is easily seen to be a bijection.  A subbasic open set in $\prod_i (X_i)_\bot$ consists of all $(x_i)_i$ such that $x_i \in U$, for some $i \in \#N$ and open $U \subseteq X_i$; the preimage of such a set is precisely $\Up U$.
\end{proof}

Each $\Up X_i = \bigcup_{j \ge i} X_j \cup \projlim_i X_i \subseteq X'$ is dense: given nonempty basic open $\Up U \subseteq X'$, with $U \subseteq X_j$ open, from the definition of $\Up U$ we must have $U \ne \emptyset$; then since $f_j, f_{j+1}, \dotsc$ have dense image, we must have $f_j^{-1}(U), f_{j+1}^{-1}(f_j^{-1}(U)), \dotsc \ne \emptyset$, whence for any $k \ge i, j$, there is some $x \in (f_j \circ \dotsb \circ f_{k-1})^{-1}(U)$, whence $x \in \Up U \cap \Up X_i$.  It follows by Baire category that $\projlim_i X_i = \bigcap_i \Up X_i \subseteq X'$ is dense.  Thus, for any nonempty open $U \subseteq X_i$, we have $p_i^{-1}(U) = \Up U \cap \projlim_i X_i \ne \emptyset$, i.e., $p_i$ has dense image.
\end{proof}

\begin{proposition}
\label{thm:opensurj-filtcolim}
Let
\begin{equation*}
\begin{tikzcd}
X_0 \dar["h_0"'] & X_1 \dar["h_1"'] \lar["f_0"'] & X_2 \dar["h_2"'] \lar["f_1"'] & \dotsb \lar["f_2"'] & X = \projlim_i X_i \lar["p_i"'] \dar["h = \projlim_i h_i"] \\
Y_0 & Y_1 \lar["g_0"] & Y_2 \lar["g_1"] & \dotsb \lar["g_2"] & Y = \projlim_i Y_i \lar["q_i"]
\end{tikzcd}
\end{equation*}
be a commutative diagram of quasi-Polish spaces, where the $h_i$ are open, the $p_i : X -> X_i$ and $q_i : Y -> Y_i$ are the limit projections, and for all $i$ and open $U \subseteq X_i$ we have the following \defn{Beck--Chevalley condition}:
\begin{align*}
h_{i+1}(f_i^{-1}(U)) = g_i^{-1}(h_i(U))
\end{align*}
(the $\subseteq$ containment is automatic).  Then $h$ is open, and satisfies for open $U \subseteq X_i$
\begin{align*}
h(p_i^{-1}(U)) = q_i^{-1}(h_i(U)).  \tag{$*$}
\end{align*}
In particular, if $h_0$ is surjective, then so is $h$.
\end{proposition}
\begin{proof}
It is enough to prove that if $h_0$ is surjective, then so is $h$.  Indeed, granting this, we may prove ($*$) (in which only $\supseteq$ is nontrivial) by truncating part of the diagram if necessary and assuming $i = 0$, then replacing $X_0$ with $U$, $Y_0$ with $h_0(U)$, $X_1$ with $f_0^{-1}(U)$, $Y_1$ with $h_1(f_0^{-1}(U)) = g_0^{-1}(h_0(U))$ (by Beck--Chevalley), $X_2$ with $f_1^{-1}(f_0^{-1}(U))$, etc., so that $X$ becomes $p_0^{-1}(U)$ and $Y$ becomes $q_0^{-1}(h_0(U))$.  Since $\projlim_i X_i$ has a basis of open sets consisting of $p_i^{-1}(U)$ for open $U \subseteq X_i$, it follows from ($*$) that $h$ is open.

So assume $h_0$ is surjective, and let $y \in Y$.  Then $h_0^{-1}(q_0(y)) \ne \emptyset$.  For each $i$, the restriction $f_i : h_{i+1}^{-1}(q_{i+1}(y)) -> h_i^{-1}(q_i(y))$ has dense image, since for open $U \subseteq X_i$ such that $U \cap h_i^{-1}(q_i(y)) \ne \emptyset$, i.e., $g_i(q_{i+1}(y)) = q_i(y) \in h_i(U)$, we have $q_{i+1}(y) \in g_i^{-1}(h_i(U)) = h_{i+1}(f_i^{-1}(U))$ by Beck--Chevalley, i.e., $f_i^{-1}(U) \cap h_{i+1}^{-1}(q_{i+1}(y)) \ne \emptyset$.  Thus by \cref{thm:dense-cofiltlim}, $h^{-1}(y) = \projlim_i h_i^{-1}(q_i(y)) \ne \emptyset$.
\end{proof}

\begin{remark}
Simpler but more abstract proofs of these results may be given using the correspondence between quasi-Polish spaces and countably presented locales \cite{Hec}.
\end{remark}

\bigskip\noindent
Department of Mathematics \\
University of Illinois at Urbana--Champaign \\
Urbana, IL 61801 \\
\medskip
\nolinkurl{ruiyuan@illinois.edu}

\end{document}